\documentclass[12pt,letter]{article}
\usepackage{amsmath}
\usepackage{ytableau}
\usepackage{amsthm,stmaryrd}
\usepackage{amsfonts, mathrsfs,amssymb,hyperref,pdfsync}
\usepackage{float}
\usepackage{amsmath,bbm}
\usepackage[numbers, sort&compress]{natbib}
\usepackage{vmargin, enumerate}
\usepackage[utf8]{inputenc}
\usepackage[normalem]{ulem}
\usepackage{pgf, tikz}
\usetikzlibrary{trees}
\usepackage{tikz}
\usetikzlibrary{automata,arrows}
\usepackage{graphicx}
\usepackage{lipsum}
\usepackage{wrapfig}
\usepackage{caption}

\newcommand{\carre}[3]{\filldraw[draw=black,fill=#3] (#1,#2)--(#1+1,#2)--(#1+1,#2+1)--(#1,#2+1)--(#1,#2)}
\newtheorem{theo}{Theorem}
\newtheorem{prop}{Proposition}
\newtheorem{lem}{Lemma}

\theoremstyle{definition}
\newtheorem{defi}{Definition}

\newcommand{\ind}{ 1\hspace{-.55ex}\mbox{l}}
\newcommand{\errtr}{\varpi}
\newcommand{\impl}{\xi}
\newcommand{\incr}{\Delta}
\newcommand{\coin}{\mathfrak{W}}
\newcommand{\azuma}{\mathfrak{H}}
\newcommand{\rounding}{\sigma}

\newcommand{\N}{\ensuremath{\mathbb N}}

\newcommand{\R}{\ensuremath{\mathbb{R}}}

\begin{document}\title{\bf  The impatient collector}

\author{Anis Amri \thanks{Institut Élie Cartan, Université de Lorraine Email: anis.amri@univ-lorraine.fr} \and Philippe Chassaing \thanks{ Institut Élie Cartan, Université de Lorraine Email: chassaingph@gmail.com}}

\maketitle
\begin{abstract} In the coupon collector problem with $n$ items, the collector needs a random number of tries $T_{n}\simeq n\ln n$ to complete the collection. Also, after $nt$ tries, the collector has secured approximately a fraction $\zeta_{\infty}(t)=1-e^{-t}$ of the complete collection, so we call $\zeta_{\infty}$ the (asymptotic) \emph{completion curve}.  In this paper, for $\nu>0$, we address the asymptotic shape $\zeta (\nu,.) $ of the completion curve under the condition $T_{n}\leq \left( 1+\nu \right) n$, i.e. assuming that the collection is \emph{completed unlikely fast}. As an application to the asymptotic study of complete accessible automata, we provide a new derivation of a formula due to Kor\v{s}unov \cite{MR517814,MR862029}.
\end{abstract}

\tableofcontents

\section{Introduction}

\subsection{Main result} 
\label{sec:mainr}

This section is intended as a concise introduction to the main results, at the price of, eventually, lacking details, for instance on the way we round real numbers where integers are expected. Details and context are given in the next section. In the standard coupon collector problem with $n$ items,  our concern is the \emph{completion curve} $\zeta_{\infty}$: after $nt$ tries, according to  \cite[pp. 4-5]{MR0471015}, the collector has secured approximately a fraction $$\zeta_{\infty}(t)=1-e^{-t}$$ of the complete collection\footnote{In \cite[pp. 4-5]{MR0471015},  the coupons outside the collection are seen as the empty cells in a random allocation scheme.}. Furthermore,   the coupon collector needs a random number of tries $T_n$ to complete the collection, with  expectation:
$$ \mathbb{E} \left[ T_{n}\right] = nH_n \sim n\ln n.$$ 
In this paper, for any given $\nu>0$, we address the asymptotic shape $\zeta_{\nu}$ of the completion curve  conditioned to the event $\mathcal{I}(\nu,n ) $ :
\begin{align*}
T_{n}&\leq (1+\nu) n,
\end{align*}
i.e.  when the collection is completed  much faster than in the classical model, hence the title. Since $(1+\nu) n=o\left(\mathbb{E} \left[ T_{n}\right]\right)$, one expects that the conditioning event $\mathcal{I}(\nu,n ) $ has an exponentially small probability, see Section \ref{largedev}. Define $\nu(N,n)=\nu$ through the relations:
$$N=(1+\nu)n,\quad \nu=\dfrac{N-n}n.$$ Formal definitions are given in the next section, but let us, for now, define the random variable $\zeta   _n(t,\omega)$ as the fraction of the complete collection secured by the collector  after $nt$ tries. Let $W_0$ denote the principal branch of the  Lambert W-function (i.e. the inverse of $x\mapsto xe^{x}$), and set:
$$F(x)=\exp\left( -1-x -W_{0}\left( -\left( 1+x \right) e^{-1-x}\right)\right).$$
Let  $\zeta(\nu,.)$ denote the unique solution, on $(0,\; 1+\nu]$, of the Cauchy problem:
\begin{equation}
\label{cp1}
y'=F\left(\tfrac{x-y}{y}\right),\;y(1+\nu)=1.
\end{equation}
The graph of $\zeta (\nu,.)$ stays in the set  $\{1+\nu\ge x>y>0\}$, and satisfies $\lim _{t\downarrow 0}\zeta \left( u,t\right) =0$, 
see Section \ref{limitpath}.
\begin{figure}[ht]
\centering
\includegraphics[width=8cm]{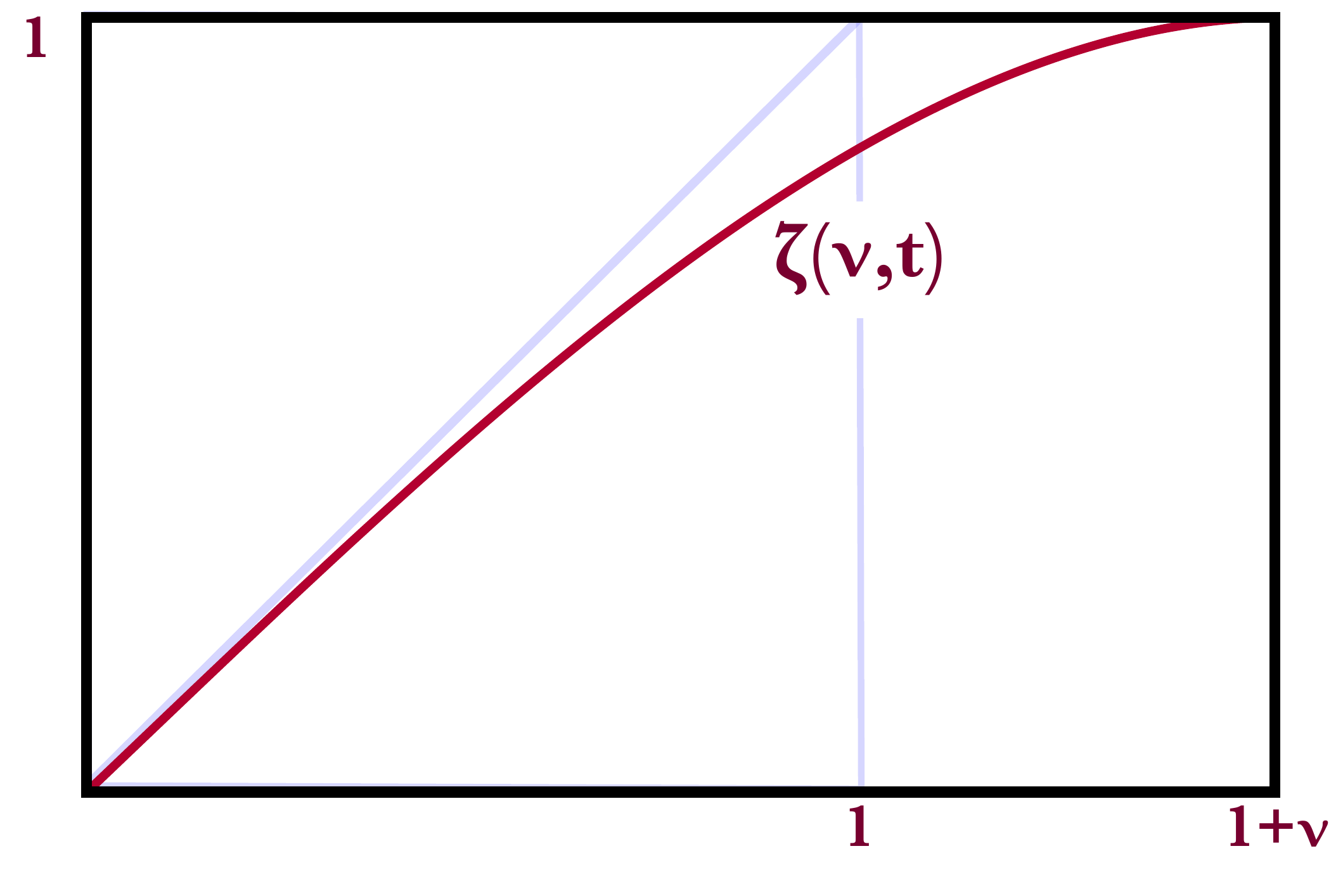}
\label{completioncurve}
\end{figure}

Let $\mathbb{P}_{N,n}$ denote the conditional probability distribution of the coupon problem, given that $T_{n}\leq  N$. The asymptotic completion curve of the impatient collector is as follows :
\begin{theo}\label{impatientprofile} For any $a,\varepsilon,\nu>0$, when $N,n\rightarrow+\infty$ with  $N/n\rightarrow 1+\nu$, i.e. with $\lim \Lambda(N,n)=\nu$, we have 
\begin{equation*}
\lim_{N,n}\mathbb{P}_{N,n} \left( \sup_{[a,1+\nu]} \left| \zeta_{n}-\zeta(\nu,.) \right|\geq \varepsilon\right) =0.
\end{equation*}
\end{theo}
\noindent Theorem \ref{impatientprofile}  is  an extension of Theorem \ref{basicprofile} (see next section), to the conditional case: $\zeta_{n}$ converges in probability to $\zeta(\nu,.)$, uniformly in any interval $[a,1+\nu]$.  Thus $\zeta(\nu,.)$  is  the $\nu$-analog of  $\zeta_{\infty}$. In the next section, we give a stronger result, in which convergence in probability is given with an explicit bound on the error,  cf. Theorem \ref{mainprofile}. In Section \ref{kornew}, we  discuss some applications of this result to random finite automata, including a new (to our knowledge) derivation of a formula by Kor\v{s}unov \cite{MR517814}. Finally, in  Sections \ref{saddle} and \ref{endix} we precise (see Theorem \ref{goodplus})  a classic asymptotic formula, due to Good, for Stirling numbers of the second kind, providing a bound that is key for our results, but could also be of independent interest.

\subsection{Context: coupon collector problem, Stirling numbers and random allocation} 
\label{sec:coupon}

Let us define more precisely the classical model (resp. the conditioned model), that we shall call  the \emph{patient model} (resp. the \emph{impatient model}). In the \emph{patient model}, we consider a sequence 
$$\omega=\left(\omega_{k}\right) _{k\ge1}$$
of uniform i.i.d. integers in $[\![1,n]\!]$. Let $\mathbb{P}_n$ denote the corresponding probability distribution on the set $[\![1,n]\!]^{\N}$ of  infinite sequences. For $\ell\ge0$, let
$$y_{\ell}\left( \omega \right) =\# \left\{ \omega_{k} \,\vline\, 1\leq k\leq \ell\right\},$$
denote the size of the collection after the $k$th try, or the number of nonempty cells after the $k$th allocation, so that $T_n$ can also be defined as follows: for $ 1\le k\le n$,
$$T_{k}\left( \omega \right) =\inf \left\{ \ell\geq 1\,\vline\, y_{\ell}\left( \omega \right) =k\right\} .$$ 
Then, set:
\[Y_n(t,\omega)=y_{\lfloor t\rfloor}\left( \omega \right),\quad t\ge0,\]
so that the \emph{completion  curve} is defined as:
\[\zeta   _n(t,\omega)=n^{-1}Y_n(nt,\omega).\]
One finds easily:
\begin{equation*}
\mathbb{E}_{n} \left[\zeta   _n(t)\right] =1- \left(1-\dfrac1n\right)^{\lfloor nt\rfloor}\simeq \zeta_\infty(t),
\end{equation*}
but also, more precisely, as a consequence of \cite[Ch. 1.1-3]{MR0471015}, 
\begin{theo}\label{basicprofile}  In the patient model, in probability, for any $t\ge 0$,
\[\lim_{n} \zeta   _n(t,.)=\zeta_{\infty}(t).\]
\end{theo}

As opposed to the patient model, in the \emph{impatient model}, we consider the conditional distribution of $\omega$ given that $T_{n}(\omega) \leq N$: then only the prefix $\omega_{[N]}=(\omega_1,\omega_2,\dots,\omega_N)$ of $\omega$ matters. In the impatient model, $\omega_{[N]}$ is  uniformly distributed on the $n!{ N\brace n }$ sequences that are surjections on  $[\![1,n]\!]$, a small subset $\Omega_{N,n}$  of $[\![1,n]\!]^N$. Here, as usual, ${m\brace \ell}$ denotes the number of partitions of a set of $m$ elements  in $\ell$ nonempty subsets, called \emph{Stirling number of the second kind}.   Thus $\mathbb{P}_{N,n}$, the conditional probability distribution of the coupon problem, given that $T_{n}\leq  N$,  is the uniform distribution on $\Omega_{N,n}$. A stronger version of Theorem \ref{basicprofile} is as follows:
\begin{theo}\label{mainprofile} For any $a>0$, and for $n_{0}$ large enough, there exists $C=C(n_{0},a)>0$ such that, for $n\ge n_{0}$,
\begin{equation*}
\mathbb{P}_{N,n} \left( \sup_{[a,1+\Lambda(N,n)]} \left| \zeta_{n}-\zeta(\Lambda(N,n),.) \right| \geq Cn^{-1 /3}\right) \leq n^{1/3}e^{-\ln ^{2}n /2}.
\end{equation*}
\end{theo}
The expression of  $C(n_{0},a)>0$ is given at Section \ref{sec:concl}.
If we assume that $\Lambda(N , n)$ stays away from 0 and $+\infty$, then, according to the asymptotic analysis of Stirling numbers of the second kind, to be found in \cite{MR0120204}, the conditioning event has an exponentially small probability:
\begin{align*}
\mathbb{P}_n \left( T_{n}\leq N\right) =
n!{ N\brace n }n^{-N}&\simeq \sqrt {\tfrac {e^{\Xi }-1}{e^{\Xi }-1-\Xi }}\ e^{-nJ \left( \Xi \right) },
\end{align*}
in which $J$ is discussed in more detail in Section \ref{largedev}. Let us just mention, now, that $\Xi=\impl(\Lambda)$ is the unique positive solution of
\begin{equation}
\label{zeta}
\impl(\Lambda) =\left( 1-e^{-\impl(\Lambda) }\right) \left( 1+\Lambda \right),
\end{equation}
that $\impl(\Lambda) =-\ln\left(F((N-n)/n)\right)$, and that $J$ is decreasing and satisfies
$$\lim _{+\infty }J \left( \Xi \right) =0,$$
which entails that $J$ is positive.  Together with
\begin{equation}
\label{defrho}
\rho=e^{-\impl},
\end{equation}
the implicit function $\impl$ is known to play a special r\^ole in the asymptotic behavior of ${ N\brace n }$, see Section \ref{saddle}.

\subsection{Asymptotics for the Stirling numbers of the second kind}
\label{sec:stirling}
First we need to set some notations. For some  integers $m\ge\ell\ge1 $, the Stirling number of the second kind, denoted by ${m\brace \ell}$, is the number  of partitions of a set of $m$ elements into $\ell $ non-empty subsets.  By convention ${0\brace 0}=0$, and for $m\geq1$ we have ${m\brace 0}=0$. Let $W_0$ denote the principal branch of the  Lambert W-function (i.e. the inverse of $x\mapsto xe^{x}$), and set:
\begin{align}
\lambda(m,\ell)&=\lambda=\frac{m-\ell}\ell,
\\
\impl  (m,\ell)&=\impl  =1+\lambda +W_{0}\left( -\left( 1+\lambda \right) e^{-1-\lambda }\right),
\\
v&=\dfrac{(\lambda+1)(\impl  -\lambda)}2.
\end{align}
We set:
$$r\left(m,\ell\right)=\dfrac{{m-1\brace \ell-1}}{{m\brace \ell}}, \quad \rho(\lambda)=e^{-\impl  }.$$
The Stirling numbers of the second kind satisfy the following recurrence relation
\begin{equation}
\label{triangle}
\forall m\ge\ell \ge0,\quad{m\brace \ell}=\ell {m-1\brace \ell}+{m-1\brace \ell-1},
\end{equation}
so that
$$\dfrac{\ell {m-1\brace \ell}}{{m\brace \ell}}=1-r\left(m,\ell\right).$$

In Section \ref{sec:transition}, we prove that for $m$, $\ell$ large, $r\left(m,\ell\right)$ depends mostly on the ratio $m/\ell$:
\begin{theo}\label{transitionp}
For any $\delta\in(0,1)$, there exist $\ell_0, C_1=C_1(\ell_0, \delta)$, both positive, such that, for any $\ell\ge\ell_0,\lambda\in(\delta, \delta^{-1})$, 
$$\left|r\left(m,\ell\right)-\rho(\lambda)\right|\le\frac{C_1}{\ell }.$$
\end{theo}
\noindent This bound proves to be crucial to our aims, for $r\left(m,\ell\right)$ and $1-r\left(m,\ell\right)$  can be seen as transition probabilities for a random walk closely related to the completion curve $\zeta   _{n}$, cf. Proposition \ref{rwstirling}. At Section \ref{sec:concl}, we describe $C_{1}$. To prove Theorem \ref{transitionp}, we need a  refinement of the asymptotic study of ${m\brace \ell}$, originally made in \cite{MR0120204}: set
\begin{equation*}
\psi(m,\ell)=\frac{1}{2\pi}\ \frac{m!}{\ell !}\ \left(\frac{e^{\impl  }-1}{\impl  ^{1+\lambda}}\right)^{\ell }\ \sqrt{\frac{\pi}{ v\ell}}.
\end{equation*}
In Good \cite{MR0120204},  $\psi$ takes the alternative form
$$\psi(m,\ell)=\frac{m!(e^{\impl  }-1)^{\ell }}{\ell !\impl  ^{m}\sqrt{2\pi m\Big(1-\frac{m}{\ell }e^{-\impl  }\Big)}}.$$

As a first step toward Theorem \ref{transitionp}, Good, followed by many others, established that $\psi(m,\ell)$
is an estimate of the corresponding Stirling number: 
\begin{theo}[\cite{MR0120204}]
\label{good}
When $\ell $ and $m$ both grow towards $+\infty$, with $m=\Theta(\ell )$,
\begin{equation*}
{m\brace \ell}\sim\psi(m,\ell).
\end{equation*}
\end{theo}
Though \cite[(3)]{MR0120204} hints at an asymptotic expansion for the relative error:
$$\chi(m,\ell)\ =\ \dfrac{{m\brace \ell}-\psi(m,\ell)}{\psi(m,\ell)},$$
it does not really provide a bound for $\chi$, while such a bound is needed to prove Theorem \ref{transitionp}. So Sections \ref{saddle} and \ref{endix}  are devoted to the proof of the following bound, of independent interest :
\begin{theo}
\label{goodplus}
 For any $\delta\in(0,1)$, there exist $\ell_0, C_2 =C_2(\ell_0, \delta)$, both positive, such that for any $\ell\ge\ell_0$,and for $\lambda(m,\ell)\in\left(\delta, \tfrac1\delta\right)$,
\begin{equation*}
\left|\dfrac{{m\brace \ell}-\psi(m,\ell)}{\psi(m,\ell)}\right|\ \le\ \frac{C_2}{\ell }.
\end{equation*}
\end{theo}

\section{The asymptotic behavior of the  completion curve}
\subsection{A random walk related to Stirling numbers }
\label{sec:rwstirling}
In this section, with the help of Theorem \ref{goodplus}, we prove Theorem \ref{mainprofile}, about the asymptotic behavior of the completion curve  of an impatient coupon collector. For a suitable elementary (small) step  $\tilde{h}$, to be defined later in the section, we shall prove that
\begin{equation}
\zeta   _{n}\left((\ell+1) \tilde{h}\right)-\zeta   _{n}(\ell \tilde{h})=\tilde{h}F(\ell \tilde{h},\zeta   _{n}(\ell \tilde{h}))+\rounding_{\ell},
\end{equation}
in which $\rounding_{\ell}=o\left(\tilde{h}\right)$, while, by definition,
\begin{eqnarray*}
\zeta   \left((\ell+1) \tilde{h}\right)-\zeta   \left(\ell \tilde{h}\right) &=\ \int_{\ell \tilde{h}}^{(\ell+1) \tilde{h}}
\ F(u,\zeta   (u))du
\\
&=\ \tilde{h}F(\ell \tilde{h},\zeta   (\ell \tilde{h}))+o\left(\tilde{h}\right).
\end{eqnarray*}
Then $\zeta   _{n}$ is the result of an Euler scheme with \emph{rounding errors} $\rounding_{\ell}$. As such,  $\zeta   _{n}$  provides a stochastic approximation for $\zeta$, in the spirit of \cite{MR1767993,MR1485774}.

Actually, time-reversed versions of $\zeta   $ and $\zeta   _{n}$, that start at time $1+\Lambda$ and end at time 0, are more convenient,  for the approximations of Stirling numbers that we use are much worse for small arguments, making the convergence trickier when $(t,\zeta   (t))$ and $(t,\zeta   _n(t))$ are close to $(0,0)$. The bound on  $\rounding_{\ell}$ is obtained through probabilistic and combinatorial tools applied to the discrete version of $\zeta   _{n}$, before it is rescaled: for any surjection $\omega$, consider a time-reversed version $Z^{(n)}$ of the completion curve  $Y_n$ of $\omega$, defined, for $t\in [0,N]   $, by
\begin{align*}
Z^{(n)}_{t}(\omega)&=Y_n(N-t,\omega).
\end{align*}
Actually the corresponding point of the curve has coordinates $W_{t}=(N-t, Z^{(n)}_t)$, and under $\mathbb{P} _{N,n} $, the probability distribution of $W = (W_k)_{k\in[\![0,N]\!]}$  has a slick description in terms of Stirling numbers of the second kind. 

\begin{figure}[ht]
\centering
\includegraphics[width=7cm]{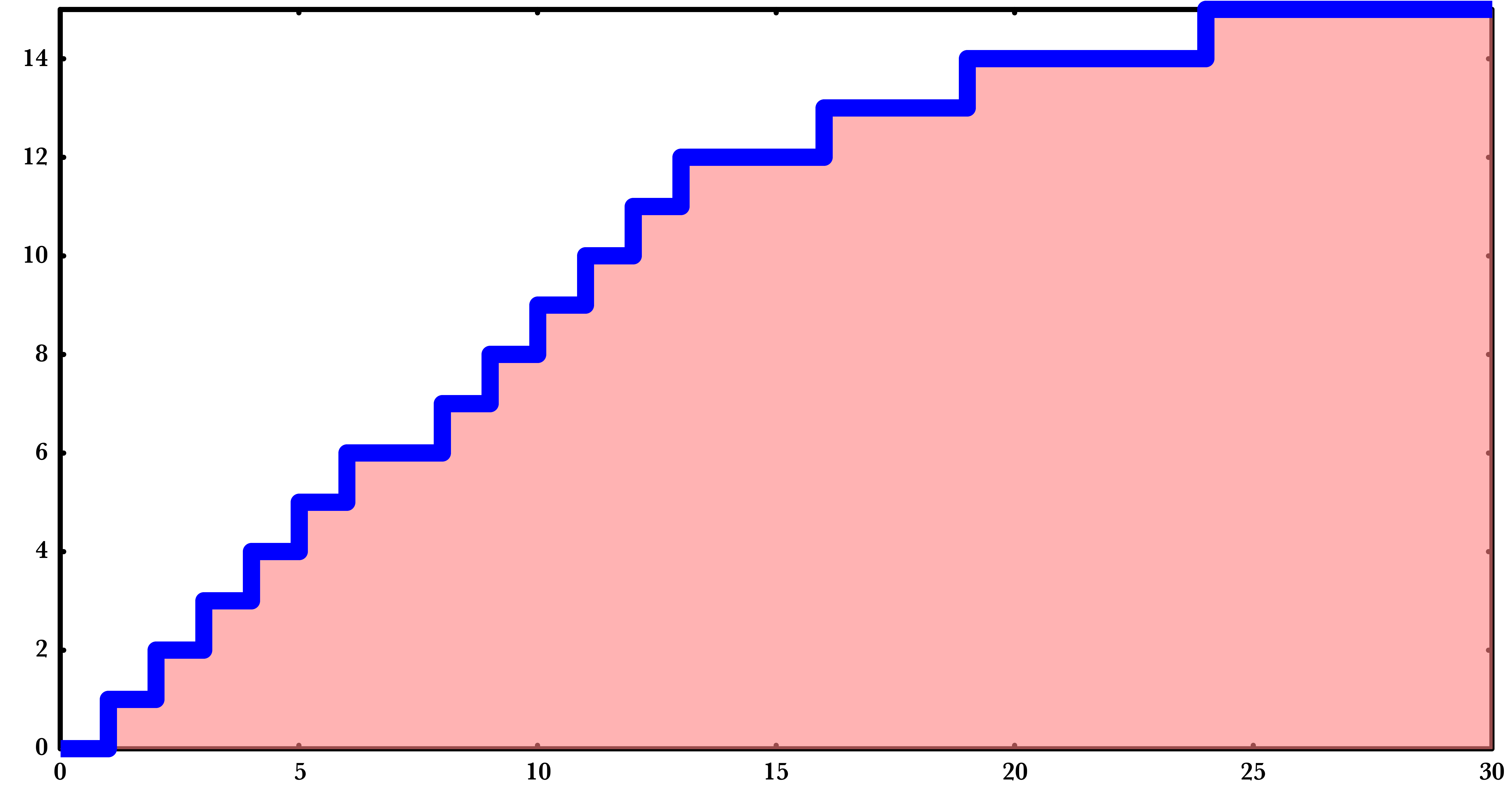} \hspace{0.5cm}\includegraphics[width=7cm]{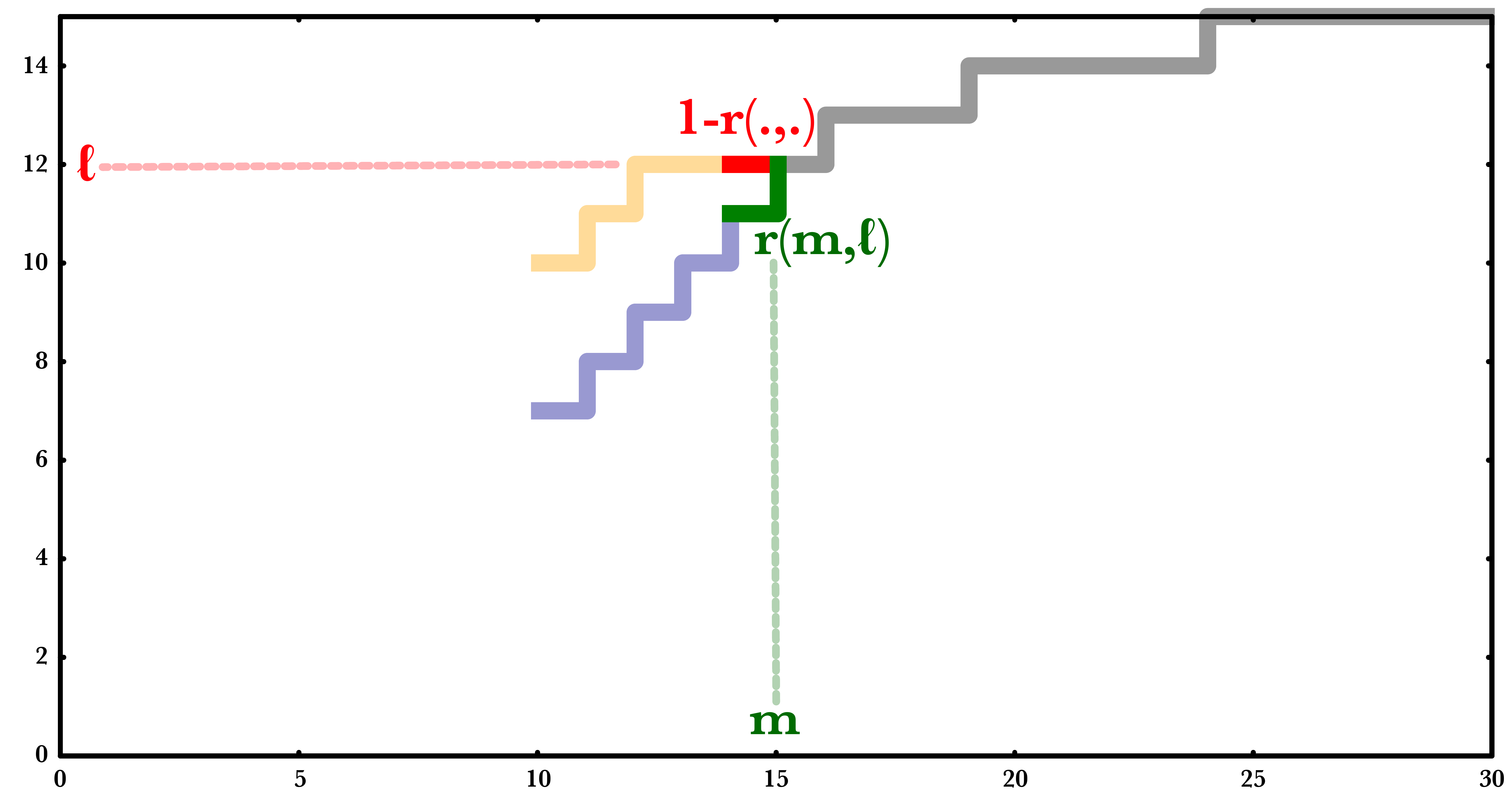}
\label{Markov}
\end{figure}

\begin{prop}
\label{rwstirling}
$W$ is a Markov chain starting at $(N, n)$, with transition probabilities described, for $0\leq \ell\leq m $, by :
$$p_{\left( m,\ell\right) ,\left( m-1,\ell-1\right)} =1-p_{\left( m,\ell\right) ,\left( m-1,\ell\right)}=\dfrac{{m-1\brace \ell-1}}{{m\brace \ell}}= r\left(m,\ell\right).
$$
\end{prop} 
In other words,  $Z^{(n)}$ is an \emph{inhomogeneous} Markov chain, with  increments $\incr_{k+1}= Z^{(n)}_{k+1  }- Z^{(n)}_k $ satisfying
\begin{equation}
	\label{bias}
	\mathbb{E } \left[ \incr_{k+1}\,\vline\, Z^{(n)}_{k}\right] =-r\left( W_{k}\right)\simeq -\rho\left( \dfrac{N-k-Z^{(n)}_k}{Z^{(n)}_k}\right).
\end{equation}
\begin{proof}
Let us compute the probability $p_z$ of a sample path 
$$z=\left( z_{0},z_{1},\ldots ,z_{N-m}\right) $$ for $Z $, in which $z_0=n $ and $z_{N-m}=\ell$: the restriction to $[\![1,m]\!]$ of any surjection $\omega$ resulting in $z $ has $\ell$ elements in its image, leading to ${m\brace \ell}n_{\downarrow\ell}$ choices for this restriction, then at each step $z_{k}\rightarrow z_{k-1}$ we have either $z_{k}$ choices for $\omega(k-1)$ if $y_{k}= z_{k}-z_{k-1}=0$, or  $n-z_{k}$ choices for $\omega(k-1)$ if $y_{k}= -1$. The second case happens $n-\ell $ times exactly, and produces a factor $n-\ell!$. Thus
\begin{align*}
	p_{z}&= {m\brace \ell}n_{\downarrow \ell}\left( \prod ^{1}_{k=N-m} \left(z_{k}\ind_{y_{k}=0}+\ind_{y_{k}\neq0}\right)\right) \left( n-\ell\ !\right) \left( {N\brace n} n!\right) ^{-1}
\\
&= {m\brace \ell}\left( \prod ^{1}_{k=N-m} \left(z_{k}\ind_{y_{k}=0}+\ind_{y_{k}\neq0}\right)\right)  {N\brace n}^{-1},
\end{align*}
while, if $z.\ell -1$ denotes the path $\left( z_{0},z_{1},\ldots ,z_{N-m}, \ell-1 \right) $—seen as a word—, we have, by the same formula, since $y _{N-m+1} =-1$ :
 $$p_{z.\ell-1 }= {m-1\brace \ell-1}\left( \prod ^{1}_{k=N-m} \left((-1+z_{k})\ind_{y_{k}=0}+1 \right)\right)  {N\brace n}^{-1}.$$
Thus the expression
$$\mathbb{P} \left( Z^{(n)}_{N-m+1}=\ell-1\,\vline\,\left( Z^{(n)}_{0},Z^{(n)}_{1},\ldots ,Z^{(n)}_{N-m}\right) =z\right)=\dfrac{p_{z.\ell-1 }}{p_{z}}=\dfrac{{m-1\brace \ell-1}}{{m\brace \ell}}=r(m,\ell)$$
depends only on the final part of the sample path, on the couple  $\left( Z^{(n)}_{N-m},Z^{(n)}_{N-m+1}\right)=(\ell,\ell-1)$. As a consequence, $W$ satisfies the Markov property, and $r(m,\ell)$ is its transition probability, as expected.\end{proof}
\subsection{Azuma inequality} 
Theorem \ref{mainprofile} is a consequence of the following chain of approximations:
\begin{align*}
dY_k=-dZ^{(n)}_{N-k}=-\incr_{N-k}
&\simeq-\mathbb{E} \left( \incr_{N-k}\right)\\ 
&=r(k,Y_k)\\
&\simeq \rho \left( \dfrac{k}{Y_{k}}-1\right) =F\left(k,Y_{k}\right) ,
\end{align*}
and its proof results from bounds for the errors in this chain of approximation, as explained before. The first error is bounded with the help of the Azuma-Hoeffding inequality, as usual when the approximation stems from the law of large numbers, while the bound for the second error, given by Theorem  \ref{transitionp}, follows from  the saddle-point method, as explained in Section \ref{saddle}. In order to use an Euler scheme, let us now divide the path into a sequence of, approximately, $\left( 1+\Lambda \right) \times n^{\beta}$ infinitesimal intervals, each of these intervals being a sequence of $h=\lfloor n^{\alpha}\rfloor$ steps, $\alpha+\beta=1, \alpha,\beta>0$.   Consider then an integer $t$ of the form $j h$, $j\in\N$, so that $t$ is the beginning of some interval, and  $t+h$ is the end of the same interval. Then 
\[Z^{(n)}_{t+h}-Z^{(n)}_{t}=A(j,h)+B(j,h)-hF\left(W_{t}\right),\]
in which:
\begin{align*}
\varepsilon_{k}&=\incr_{k+1}-\mathbb{E}\left[\incr_{k+1}\mid Z^{(n)}_{k}\right],
\\
A(j,h)&=\sum\limits_{s=0}^{h-1}\varepsilon_{t+s},
\\
B(j,h)&=\sum\limits_{s=0}^{h-1}\left(\mathbb{E}\left[\incr_{t+s+1}\mid Z^{(n)}_{t+s}\right]+F\left(W_{t}\right)\right)
\\
&=\sum\limits_{s=0}^{h-1}\left(F\left(W_{t}\right)-r\left(W_{t+s}\right)\right),
\end{align*}
the last equality due to \eqref{bias}. Rescaling time and space by a factor $1/n$, we set $\tilde{h}=h/n$ and 
$$\rounding_{j}=n^{-1}\left(A(j,h)+B(j,h)\right).$$
Finally, for $\eta, \delta\in (0,1) $,  we set
$$\coin_{\eta, \delta}=\left\{ \left( x,y\right),\,x > \eta,\,\lambda\left( x,y\right)\in \left( \delta ,\delta^{-1}\right) \right\} =\left\{ \left( x,y\right) ,\,x > \eta,  \tfrac{\delta x}{1+\delta} \le y\le \tfrac{ x}{1+\delta} \right\},$$ 
in such a way that, according to Theorem \ref{goodplus},  $|\ell\chi(m,\ell)|$ is uniformly bounded for $(m,\ell)$ in $n\coin_{\eta, \delta}$, as long as $n$ is large enough, and the same holds true for Theorem \ref{transitionp}. Now, for $x\in[\eta,1+\Lambda]$, by geometric considerations,
\begin{align}
\label{recurse}
\left\{\left( x,y\right) \in\coin_{\eta, 2\delta}\ \textrm{and}\ \left\vert y-z\right\vert\le \dfrac{\delta\eta}6\right\}\Rightarrow\left\{\left( x,z\right) \in\coin_{\eta, \delta}\right\}.
\end{align}
Section \ref{sec:stirling} entails that
\begin{lem}
\label{lemmalipsch}
For $n$ large enough, and for $\eta, \delta\in (0,1) $,  if $n^{-1}W_{jh}\in \coin_{\eta, 2\delta}$ and $N-(j+1)h\ge\eta n$, we have 
\begin{equation}
\label{round2}
n^{-1}B(j,h)\le \dfrac {8}{\eta}\   n^{2\alpha-2}.
\end{equation}
\end{lem}
\begin{proof}  Recall that $t=j  h$. If
$$n^{\alpha -1}\leq \dfrac {\eta\delta }{6 }, $$
then $$\{n^{-1}W_{t}\in \coin_{\eta, 2\delta}\}\Rightarrow \{\forall s\in[\![1,h]\!],\ n^{-1}W_{t+s}\in \coin_{\eta, \delta}\},$$
but, if $n^{-1}W_{t+s}\in \coin_{\eta, \delta}$, we obtain, below, that
\begin{equation}
\label{lipsch}
\left|r\left(W_{t+s}\right)-F\left(W_{t}\right)\right|\le \dfrac {8}{\eta}\ n^{\alpha-1},
\end{equation}
entailing \eqref{round2}.
Relation  \eqref{lipsch} follows from the Taylor inequality for $\rho$,  provided that  both $n^{-1}W_{t}$ and $n^{-1}W_{t+s}$ belong to $ \coin_{\eta, \delta}$:
\begin{align*}
\left|r\left(W_{t+s}\right)-F\left(W_{t}\right)\right|
&\le
\left|r\left(W_{t+s}\right)-F\left(W_{t+s}\right)\right|
+ \left|F\left(W_{t+s}\right)-F\left(W_{t}\right)\right|
\end{align*}
and, since $W_{t+s} $ meets the conditions in Theorem \ref{transitionp},
\begin{align*}
\left|r\left(W_{t+s}\right)-F\left(W_{t+s}\right)\right|&\le \frac{C(\ell_0,\delta)}{N-t-s}\le \frac{C(\ell_0,\delta)}{\eta n}
\end{align*}
while, according to   Section  \ref{sec:deltarho},
\begin{align*}
\left|F\left(W_{t+s}\right)-F\left(W_{t}\right)\right|&=\left|\rho\left(\tfrac{Z^{(n)}_{t+s}}{N-t-s}-1\right)-\rho\left(\tfrac{Z^{(n)}_t}{N-t}-1\right)\right|\\&\le \dfrac {4s}{\eta n}\le \dfrac {4}{\eta}\ n^{\alpha-1}.
\end{align*} 
For $n$ large enough, 
$$\dfrac {4}{\eta}\ n^{\alpha-1}+  \frac{C(\ell_0,\delta)}{\eta n}\le \dfrac {8}{\eta}\ n^{\alpha-1},$$
yielding  successively \eqref{lipsch}, then \eqref{round2}. 
\end{proof}
Also,  for $t,k\ge 0$, let $\mathcal{F}_k$ denote the $\sigma$-algebra $\sigma(Z_{1},\;Z_2,\ldots,Z_{t+k})$, and
\begin{align*}
\mathcal{M}_k&=\sum\limits_{s=0}^{k-1}\varepsilon_{t+s}.
\end{align*}
The sequence $(\mathcal{M}_k)$ is a martingale with respect to the filtration $\mathcal{F}$ and for any $k$, $|\mathcal{M}_{k+1}-\mathcal{M}_k|\leq1$, thus  Azuma's inequality gives :
\begin{equation*}
\mathbb{P}(|\mathcal{M}_h|\geq u)\leq e^{-\tfrac{u^2}{2h}},
\end{equation*}
in which $\mathcal{M}_{h}=A(j,h)$. For $u=n^{\alpha/2}\ln n$,  we  obtain:
\begin{equation*}
\mathbb{P}\left(n^{-1}A(j,h)\geq u / n\right)\leq e^{-\tfrac{\ln^{2}n}{2}}.
\end{equation*}
Set $$\azuma_n= \left\{\omega \in \Omega ,\exists j\in \left[ 0,\left( 1+\Lambda \right) n^\beta \right]\ \textrm{such~that}\ \left| A\left( j,h,\omega \right) \right| \geq n^{\alpha /2}\ln n\right\}. $$ The previous bounds lead to
\begin{prop}
For $n$ large enough, the  set $\azuma_n$ satisfies:
$$ \mathbb{P} \left( \azuma_n\right) \leq 2\left( 1+\Lambda \right) n^{\beta }e^{-\ln ^{2}n/2}.  $$
\end{prop}

\subsection{Euler scheme}

Thus, for $\omega\notin\azuma_n$, i.e. but for a probability at most $\mathcal{O}\left(n^{\beta}e^{-\tfrac{\ln^{2}n}2}\right)$, $\left(\zeta   _{n}(t)\right)_{0\le t\le k}$ is obtained through an Euler scheme with step $\tilde{h}=n^{-1}h\simeq n^{\alpha-1}$ and rounding error $\rounding_{j}$ such that
\begin{align*}
\vert \rounding_j\vert&=\vert n^{-1}(A(j)+B(j))\vert
\\
&\le n^{-1+\alpha/2}\ln n+ \dfrac {8}{\eta}\   n^{2\alpha-2}.
\end{align*}
For the choice $\alpha=2 /3$,  $\beta=1 /3$, and for $n$ large enough, depending on the choice of $(\ell_0,\eta,\delta)$, we obtain that
\begin{align*}
\vert \rounding_j\vert&\le 2n^{-2/3}\ln n.
\end{align*}
Then we can see $\zeta   (\ell \tilde{h})$, resp. $\zeta   _{n}(\ell \tilde{h})$,  as the solution of the ODE at time $\ell \tilde{h}$ (resp. the output of the Euler scheme after $\ell$ steps), and set 
$$e_{\ell}=\vert\zeta   _{n}(\ell \tilde{h})-\zeta   (\ell \tilde{h})\vert.$$
Then, following \cite{MR0202267} and according to Section \ref{sec:deltarho}, provided that the points $M_{n,\ell}=\left(\ell \tilde{h},\zeta   _{n}(\ell \tilde{h})\right)$ and  $M_\ell=\left(\ell \tilde{h},\zeta   (\ell \tilde{h})\right)$ belong to $\coin_{\eta, 2\delta}$, and that $1+\Lambda-(\ell+1) \tilde{h}\ge\eta$, we can write 
\begin{align}
\nonumber
e_{\ell+1}&\le \left\vert\zeta   _{n}\left((\ell+1) \tilde{h}\right)-\zeta   _{n}(\ell \tilde{h})+\zeta   (\ell \tilde{h})-\zeta   \left((\ell+1) \tilde{h}\right)\right\vert+e_{\ell}
\\
\nonumber
&=\left\vert \rounding_{\ell}+\tilde{h}F(\ell \tilde{h},\zeta   _{n}(\ell \tilde{h}))+\zeta   (\ell \tilde{h})-\zeta   \left((\ell+1) \tilde{h}\right)\right\vert+e_{\ell}
\\
\nonumber
&=\left\vert \rounding_{\ell}+\tilde{h}F(\ell \tilde{h},\zeta   _{n}(\ell \tilde{h}))-\tilde{h}F(\ell \tilde{h},\zeta   (\ell \tilde{h}))-\dfrac{\tilde{h}^{2}}2\zeta   ^{\prime\prime}((\ell+\theta) \tilde{h}))\right\vert+e_{\ell}
\\
\nonumber
&\le\vert \rounding_{\ell}\vert+\tilde{h}\left\vert F(\ell \tilde{h},\zeta   _{n}(\ell \tilde{h}))-F(\ell \tilde{h},\zeta   (\ell \tilde{h}))\right\vert+\dfrac{\tilde{h}^{2}}2\left\vert\zeta   ^{\prime\prime}((\ell+\theta) \tilde{h}))\right\vert+e_{\ell}
\\
\label{boundsegonde}
&\le\vert \rounding_{\ell}\vert+\tilde{h}e_{\ell}\sup_{u\in[\zeta   _{n}(\ell \tilde{h}),\zeta   (\ell \tilde{h})]}\left\vert F^{\prime}_{y}(\ell \tilde{h},u)\right\vert+\dfrac{\tilde{h}^{2}}2\sup_{v\in[\ell\tilde{h},(\ell+1) \tilde{h})]}\left\vert\zeta   ^{\prime\prime}(v)\right\vert+e_{\ell}
\\
\nonumber
&\le  e_{\ell}\left( 1+\dfrac {2\tilde{h} }{\eta}\right) +\dfrac{\tilde{h}^{2}}{2\eta}+\vert \rounding_{\ell}\vert,
\\
\label{recursefin}
&=  e_{\ell}\left( 1+K\tilde{h}\right) +u_{\ell},
\end{align}
in which 
$$\vert \rounding_{\ell}\vert\le n^{-2/3}\ln n,\quad  \tilde {h}\le 2n^{-1/3},\quad u_\ell\le  n^{-2/3}\left(1/\eta+\ln n\right),$$
and
$$K=\dfrac {2 }{\eta}.$$ The bounds for the supremums  in \eqref{boundsegonde} are obtained in Section \ref{sec:transition}, see Proposition \ref{prop-deltarho}.
For $\ell=0$, $M_{n,0}=M_0=\left(1+\Lambda,1\right)\in\coin_{a, 4\delta}$ for $\delta$ small enough.  Consider the bound \eqref{lambdabound} obtained for $\lambda$ at  Section  \ref{limitpath}. It entails that, for $x\in[a,1+\Lambda]$,
\begin{align}
\label{lambdaboundxi}
\dfrac {a\Lambda}{(1+\Lambda)^{2}}  &\leq\lambda \left( x,\zeta   (x)\right) \leq \Lambda,
\end{align}
so that $\left( x,\zeta   (x)\right) \in\coin_{a, 4\delta}$ for $4\delta\le \min\left(\tfrac{a\Lambda}{(1+\Lambda)^{2}},\Lambda^{-1} \right)$. Thus $M_\ell\in\coin_{a, 4\delta}\subset\coin_{a, 2\delta}$ if $\ell \tilde{h}\in[a,1+\Lambda]$. Now, for $x\in[a,1+\Lambda]$,
\begin{align}
\label{recurse}
\left\{\left( x,\zeta   (x)\right) \in\coin_{a, 4\delta}\ \textrm{and}\ \left\vert\zeta   (x)-\zeta   _n(x)\right\vert\le \dfrac{a\delta}3\right\}\Rightarrow\left\{\left( x,\zeta   _n(x)\right) \in\coin_{a, 2\delta}\right\}.
\end{align}
Assume that, for $k\le\ell$, $M_{n,k-1}\in\coin_{a, 2\delta}$, so that we can write :
\begin{align}
\nonumber
e_{k }&\le e_{k -1}(1+K\tilde{h})+(\tilde{h}^{2}/2a)+\vert \rounding_{k }\vert
\\
\nonumber
&= e_{k -1}(1+K\tilde{h})+u_k 
\\
\nonumber
&\le e_{k -2}(1+K\tilde{h})^2+u_k +(1+K\tilde{h})u_{k -1}
\\
\nonumber
&\le e_{k -3}(1+K\tilde{h})^3+u_k +(1+K\tilde{h})u_{k -1}+(1+K\tilde{h})^2u_{k -2}
\\
\nonumber
&\le e_{0}(1+K\tilde{h})^k +u_k +(1+K\tilde{h})u_{k -1}+\dots+(1+K\tilde{h})^k  u_{0}
\\
\nonumber
&= u_k +(1+K\tilde{h})u_{k -1}+\dots+(1+K\tilde{h})^k  u_{0}.
\end{align}
Then
\begin{align}
\nonumber
e_{\ell}&\le\frac{(1+K\tilde{h})^{\ell+1}-1 }{K\tilde{h}}\ n^{-2/3}\left(1/a+\ln n\right),
\\
\nonumber
&\le 2\ \frac{ 2e^{K(1+\Lambda)}-1 }{K}\ n^{-1/3}\left(1/a+\ln n\right),
\\
\label{boundfinale}
&\le \frac{ 8e^{K(1+\Lambda)}}{K}\ n^{-1/3}\ \ln n\le \dfrac{a\delta}3,
\end{align}
for $n$ large enough, depending on $(\delta,a)$, but not on $\ell$, since : 
$$\left( 1+K\tilde{h}\right) ^{\ell}\leq e^{K\tilde{h}\ell}\leq e^{K(1+\Lambda)},$$
for $\ell\le (1+\Lambda) n^\beta\simeq(1+\Lambda)\tilde{h}^{-1}$. 
Relations \eqref{recurse} and \eqref{boundfinale} entail that $M_{n,\ell}\in\coin_{a, 2\delta}$ so that \eqref{recursefin} holds true and, in turn, $M_{n,\ell+1}\in\coin_{a, 2\delta}$, if necessary. It follows, recursively,  that, for any $\ell\le (1+\Lambda) n^\beta$,
$$e_\ell\le c n^{-1/3}\ln n,\quad c=\frac{ 8e^{K(1+\Lambda)}}{K},$$
that is, at the ends of any infinitesimal interval, the error $\left| \zeta    _{n}-\zeta    \right| $ is bounded accordingly. Between these ends the error can be larger by at most  half the length of this infinitesimal interval, i.e. by $n^{-1 /3}/2$, since  both $  \zeta    _{n}$ and $\zeta   $ are non increasing with slope smaller than 1. Finally, for $n$ large enough and $\omega\notin\azuma_n$, i.e. but for a probability at most $\mathcal{O}\left(n^{-\ln n/2\ +\beta}\right)$, on the interval $[a,1+\Lambda]$, 
$$\Vert\zeta   (\omega)-\zeta   _{n}(\omega)\Vert_{\infty}\le (1+c \ln n)n^{-1/3}.$$

\section{Coupon and automata}
\label{kornew}

\subsection{Kor\v{s}unov's formula}
In $1978$, Kor\v{s}unov \cite{MR517814,MR862029}  proves a formula for the asymptotic enumeration of  accessible complete and deterministic automata (ACDA) with $n$ states over a $k$-letters alphabet. Later Nicaud \cite{Nicaudth}  proves that  ACDA are in bijection with  a subset $\mathcal{A}_{k, n}$ of $\Omega_{kn+1,n}$, though he uses a different terminology  : surjections are represented by \emph{boxed diagrams}, and ACDA by \emph{Dyck} boxed diagram. We recall briefly the definitions of these combinatorial objects in the next subsection. In this paper we assume that $k\ge 2$ and we set $N=kn+1$. With these notations, we can rephrase Kor\v{s}unov's result as follows :
\begin{theo}\cite{MR517814,MR862029}
\label{kor78}
\begin{equation*}
\lim_{n}\mathbb{P}_{N,n} \left(\mathcal{A}_{k, n} \right)=1-k\rho(k)>0.
\end{equation*}
\end{theo}
In the notations of \cite{MR2786470}, $1-k\rho(k)=(1-\rho(k))E_k$. In Section \ref{basics}, we describe  $\mathcal{A}_{k, n}$ following the lines of \cite{Nicaudth}, then in Sections \ref{NEcorner},\ref{pollaczek},\ref{tails} we give a probabilistic proof of Theorem \ref{kor78} : with the help of Theorem \ref{mainprofile} and of the representation of ACDA, taken from  \cite{Nicaudth},  Theorem \ref{kor78} reduces to  the  Pollaczeck-Khinchine formula for a simple random walk. In Section \ref{automataprofile} we  explain  how Theorem \ref{mainprofile}  extends to  ACDA.

\subsection{Basics on automata}
\label{basics}

In this section, we recall briefly some vocabulary on words and automata, taken from \cite[Section 1.3]{MR2165687}, then we describe the  representation of ACDA by boxed diagrams, following \cite{Nicaudth}. Let $\mathcal{A}$ be a finite totally ordered set, called \emph{alphabet}. The elements of $\mathcal{A}$  are called \emph{letters} or also \emph{symbols}. A finite  \emph{word} $w$ on the alphabet $\mathcal{A}$ is a finite sequence   $w=w_1w_2\ldots w_n$ of elements of $\mathcal{A}$. The set of words is endowed with the operation of  \emph{concatenation}, also called product, in which two words $u=u_1u_2\ldots u_p$ and $v=v_1v_2\ldots v_q$ give the word $uv=u_1u_2\ldots u_pv_1v_2\ldots v_q$. This operation is associative, and it has a neutral element, the  \emph{empty word}, denoted by $\emptyset$. The  \emph{length} of a word $u$,  denoted  $\vert u\vert$, is the number of letters in the word $u$ (so that $\vert \emptyset\vert=0$). We denote by $\mathcal{A}^*$ the set of finite words on the alphabet $\mathcal{A}$.

\begin{defi}
A \emph{deterministic and complete automaton} $\mathfrak{A}$ is a quintuplet $(\mathcal{A},Q,\delta,I,F)$ consisting of:
\begin{itemize}
\item an alphabet $\mathcal{A}$, such that $\#\mathcal{A}=k$,
\item a set  $Q$ of states, such that $\#Q=n$,
\item an  \emph{initial  state} $q_0$,
\item a \emph{transition function}, $\delta$, that takes as argument a state and a symbol and returns a state,  $\delta:Q\times \mathcal{A} \rightarrow Q$,
\item a set of \emph{final  states} $F\subset Q$.
\end{itemize}
\end{defi}
The transition function $\delta$ has a straightforward extension to $Q\times \mathcal{A}^*$, that describes a path from a state $q$ to another state  $\delta(q,w)$ through a sequence $w$ of letters (=edges) in a directed graph related to $\delta$, see the figure below. For instance, for $w=w_1w_2\in \mathcal{A}^2 $,
$$\delta \left( q,w_{1}w_{2}\right) =\delta \left( \delta \left( q,w_{1}\right) ,w_{2}\right). $$
\begin{defi}
A deterministic finite automaton $\mathfrak{A}$ is \emph{accessible} when for each state $q$ of $\mathfrak{A}$, there exists a word $u\in \mathcal{A} ^*$ such that $\delta(q_0,u)=q$. 
\end{defi}
\begin{defi}
A word $u$ is \emph{recognized} by an automaton when $\delta(q_0,u)\in F$. The \emph{language} recognized by an automaton is the set of words that it recognizes. 
\end{defi}

\textbf{Two representations of an ACDA}.
Consider the  ACDA  $\mathfrak{A}=(\mathcal{A},Q,\delta,I,F)$ given by the alphabet $\mathcal{A} =\{a,b,c\}$, $Q=\{q_0,q_1,q_2,q_3\}$, $I=q_0$ and $F=\{q_3\}$. The  transition table of $\mathfrak{A}$ is a first representation of $\delta$, for instance :
 \begin{equation*}
   \begin{tabular}{ l ||c | c | r|}
     \hline
    $\delta$  & $a$ & $b$ &$c$ \\ \hline
  $\rightarrow$ $q_0$ & $q_1$& $q_2$ &$q_3$\\ \hline
  \quad $q_1$ & $q_1$ &$ q_1$ &$q_3$\\ \hline
   \quad $q_2$ & $q_2$ & $q_2$ &$q_1$\\ \hline
   * $q_3 $&  $q_3 $&  $q_3$& $q_3$\\
     \hline
   \end{tabular}
 \end{equation*}
The symbol $\rightarrow$ marks the initial state, here it is $ q_0$. The symbols $*$ mark the final state(s) (here there is only one final state, $ q_3$). 

Another representation is through a directed graph with edges labelled by $\mathcal{A} $ and whose  set of vertices is $Q$ : for $(a,q,r)\in \mathcal{A} \times Q^2$ a directed edge $(q,r)$ with label $a$ is present in the graph of $\mathfrak{A} $ if and only if $\delta(q,a)=r$. Then each vertex of the graph has out-degree $k$, and there is a path from $q$ to $r$ in the graph if and only if there exists a word  $u\in \mathcal{A} ^*$ such that $\delta(q,u)=r$, hence the term \emph{accessible}. Only the initial state has an ingoing edge with no starting point, and the final states have an outgoing edge with no endpoint.
 \begin{center}
 \begin{tikzpicture}[>=stealth',shorten >=1pt,auto,node distance=2 cm, scale = 1, transform shape]
\node[initial,state] (A)                                    {$q_0$};
\node[state]         (B) [right of=A]                       {$q_1$};
\node[state]         (C) [above of=B]                       {$q_2$};
\node[state,accepting]         (D) [below of=B]                       {$q_3$};

\path[->] (A) edge [left]       node [align=center]  {$ b  $} (C)
      (A) edge [above]      node [align=center]  {$ a$} (B)
      (A) edge [left]       node [align=center]  {$ c$} (D)
      (B) edge [right]      node [align=center]  {$ a$} (D)
      (B) edge [ loop right ]   node [align=center]  {$ b,c$} (B) 
      (C) edge [loop above] node [align=center]  {$ a,b  $} (C)
       (C) edge [right]      node [align=center]  {$ c $} (B)     
        (D) edge [loop below] node [align=center]  {$a,b, c $} (D);

\end{tikzpicture}
\end{center}
The accessibility of an automaton $\mathfrak{A}=(\mathcal{A},Q,\delta,q_0,F)$  depends only on its \emph{transition structure} $\mathfrak{D}=(\mathcal{A},Q,\delta,q_0)$, not on its final states, thus one  can discuss the  accessibility of a complete deterministic transition structure (CDTS) :  $(\delta,q_0)$ can be seen as a map $\delta^*$ from the set of edges $\{\rightarrow\}\cup\left(Q\times\mathcal{A} \right)$, including thus $Q\times\mathcal{A}$ plus the starting edge $\rightarrow$, such that  $\delta^*\left(\rightarrow\right)=q_0$ and $\delta^*\vert_{Q \times\mathcal{A} }=\delta$. The CDTS is accessible only if its transition function $\delta^*$ is a surjection,  that is,  $\delta^*$ has to belong to $\Omega_{N,n}$ and this has to be the connection between the impatient collector and ACDA. However, two problems arise: 
\begin{itemize}
  \item though $\{\rightarrow\}\cup\left(Q\times\mathcal{A} \right)$ has $kn+1=N$ elements, a total order would be handy to identify  $\{\rightarrow\}\cup\left(Q\times\mathcal{A} \right)$ with $[\![1, N ]\!]$, and  $\delta^*$ with an element of $\Omega_{N,n}$ ;
  \item the surjectivity of $\delta^*$ is not sufficient to insure the connexity of $\mathfrak{D}$.
\end{itemize}
It turns out that the answer to the second point is also an answer to the first point : as usual for the connexity of graphs, a necessary and sufficient condition of connexity is a positivity condition for a path related to the breadth-first-search of the corresponding graph, and this breadth-first search also provides a total order on $Q$, which, with the  alphabetic order on $\mathcal{A}$, induces a lexicographic order on  $\{\rightarrow\}\cup\left(Q\times\mathcal{A} \right)$, allowing to identify   $\{\rightarrow\}\cup\left(Q\times\mathcal{A} \right)$ with  $[\![1, N ]\!]$. The path, then, is the completion curve for $\delta^*$ once   $\{\rightarrow\}\cup\left(Q\times\mathcal{A} \right)$ is identified to $[\![1, N ]\!]$. 

More precisely, the search starts from the initial vertex,    $\delta^*(\rightarrow)$,  of the first directed edge $\rightarrow$, end vertex relabelled $1$ for it is the first piece in the collection. Then one explores the edges starting from $q_0=1$, in the alphabetic order, from $\delta(q_0,a_1)$ to $\delta(q_0,a_n)$, and when this exploration is over, either there exists no new piece in the collection, meaning that $q_0$ is a connected component by itself, and meaning that $\mathfrak{D}$ is not accessible, or there exists some new piece. Thus the completion curve of an  accessible CDTS   must satisfy $y_{k+1}\ge 2$. The $y_{k+1}-1$ new vertices, at this stage, are of the form
$$\delta^*(\delta^*(\rightarrow),a)=\delta^*(\rightarrow a),$$
 and they are sorted (and explored) according to the alphabetic order for the letters $a\in\mathcal{A}$. They are also relabeled $2,3,\dots,y_{k+1}$. Similarly, after the exploration of the neighbours of the second piece $q_1=2$ of the collection, we need $y_{2k+1}\ge 3$, else $\{q_0,q_1\}$ would be a connected component. In general, the CDTS is accessible if and only if the completion curve $y$ of $\delta^*$ satisfies 
\begin{equation}
\label{connexion}
	\forall \ell\in[\![0,n-1]\!],\quad y_{\ell k+1}\ge \ell+1.
	\end{equation}
Now, according to \cite{Nicaudth}, the boxed diagram of $\mathcal{D}$ is just the completion curve of $\delta^\ast$, decorated with one mark in each column, at height $x_i\le y_i$, meaning that $\delta^*(i)=x_i$.

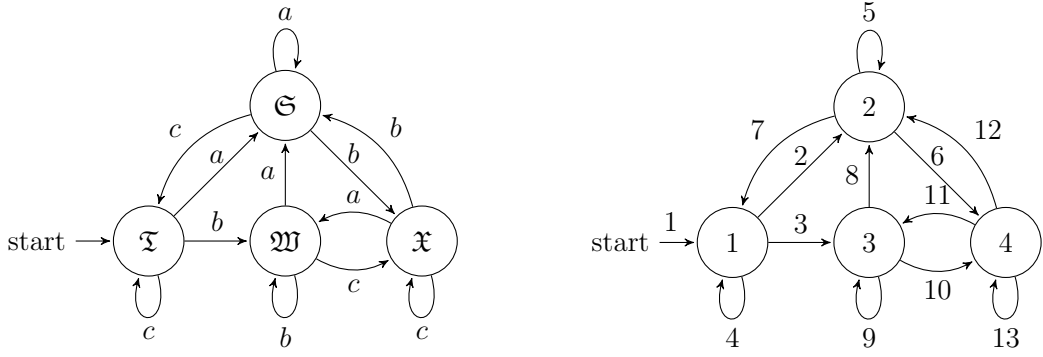
\begin{figure}
\begin{center}
\begin{tikzpicture}[>=stealth',shorten >=1pt,auto,node distance=2 cm, scale = 0.9, transform shape]
\node[initial,state] (A)                                    {$\mathfrak{T}$};
\node[state]         (B) [right of=A]                       {$\mathfrak{W}$};
\node[state]         (C) [above of=B]                       {$\mathfrak{S}$};
\node[state,]         (D) [right of=B]                       {$\mathfrak{X}$};
\path[->] (A) edge [loop below]       node [align=center]  {$ c  $} (A)
(A) edge [above]       node [align=center]  {$ b$} (B)
(A) edge [above]       node [align=center]  {$ a$} (C)
(B) edge [loop below]       node [align=center]  {$ b $} (B)
 (C) edge [loop above] node [align=center]  {$ a  $} (C)
 (D) edge [loop below]      node [align=center]  {$ c$} (D) 
 (C) edge [above]       node [align=center]  {$ b$} (D)
  (B) edge [left]       node [align=center]  {$ a$} (C) 
   (C) edge [bend right]     node [swap] {$ c$} (A) 
   (D) edge [bend right]       node [swap]  {$ b$} (C) 
   (D) edge [bend right]       node [swap]  {$a$} (B) 
   (B) edge [bend right]       node [swap]  {$c$} (D)  ;   
\end{tikzpicture}\hspace{1.5cm}
\begin{tikzpicture}[>=stealth',shorten >=1pt,auto,node distance=2 cm, scale = 0.9, transform shape]
\node[state,initial] (A)                                    {$1$};
\node[state]         (B) [right of=A]                       {$3$};
\node[state]         (C) [above of=B]                       {$2$};
\node[state,]         (D) [right of=B]                       {$4$};

\path[->] 
(A) edge [loop below]       node [align=center]  {$4$} (A)
(A) edge [above]       node [align=center]  {$3$} (B)
(A) edge [above]       node [align=center]  {$2$} (C)
(B) edge [loop below]       node [align=center]  {$9$} (B)
 (C) edge [loop above] node [align=center]  {$5$} (C)
 (D) edge [loop below]      node [align=center]  {$ 13$} (D) 
 (C) edge [above]       node [align=center]  {$6$} (D)
  (B) edge [left]       node [align=center]  {$8$} (C) 
   (C) edge [bend right]     node [swap] {$7$} (A) 
   (D) edge [bend right]       node [swap]  {$12$} (C) 
   (D) edge [bend right]       node [swap]  {$11$} (B) 
   (B) edge [bend right]       node [swap]  {$10$} (D)  ;   
 \node at (-0.9,0.3) {$ {1}$};
\end{tikzpicture}
\end{center}
\caption{On the left: a CDTS $\mathcal{D}$. On the right: the ordering of vertices and edges inherited from the breadth-first search.}
\label{canon}
\end{figure}

\begin{figure}
\centering
\begin{tikzpicture}[scale=0.6]
\foreach \i in {5,...,12}
{\foreach \j in {1,...,4}
{\carre{\i}{\j}{white};}}
\foreach \i in {2,...,5}
{\foreach \j in {1,...,3}
{\carre{\i}{\j}{white};}}
\foreach \i in {1,...,2}
{\foreach \j in {1,...,2}
{\carre{\i}{\j}{white};}}
\foreach \i in {0,...,1}
{\foreach \j in {1,...,1}
{\carre{\i}{\j}{white};}}
\node at (0.5,1.5) {\textcolor{red}{$\times$}};
\node at (1.5,2.5) {\textcolor{red}{$\times$}};
\node at (2.5,3.5) {\textcolor{red}{$\times$}};
\node at (3.5,1.5) {$\times$};
\node at (4.5,2.5) {$\times$};
\node at (5.5,4.5) {\textcolor{red}{$\times$}};
\node at (6.5,1.5) {$\times$};
\node at (7.5,2.5) {$\times$};
\node at (8.5,3.5) {$\times$};
\node at (9.5,4.5) {$\times$};
\node at (10.5,3.5) {$\times$};
\node at (11.5,2.5) {$\times$};
\node at (12.5,4.5) {$\times$};
\draw[very thick, blue] (0,1)--(0,2)--(1,2)--(1,3)--(2,3)--(2,4)--(5,4)--(5,5)--(13,5);
\draw[thick,red] (0.05,0.95)--(0.05,1.95)--(3,1.95)--(3,3)--(6,3)--(6,4)--(9,4)--(9,4.95)--(13,4.95);
\fill[red, opacity=0.1] (0,1)--(0,2)--(3,2)--(3,1);
\fill[red, opacity=0.1] (3,1)--(3,3)--(6,3)--(6,1);
\fill[red, opacity=0.1] (6,1)--(6,4)--(9,4)--(9,1);
\fill[red, opacity=0.1] (9,1)--(9,4.95)--(13,4.95)--(13,1);
\node at (12.5,0.5){$13$};
\node at (11.5,0.5){$12$};
\node at (10.5,0.5){$11$};
\node at (9.5,0.5){$10$};
\node at (8.5,0.5){$9$};
\node at (7.5,0.5){$8$};
\node at (6.5,0.5){$7$};
\node at (5.5,0.5){$6$};
\node at (4.5,0.5){$5$};
\node at (3.5,0.5){$4$};
\node at (2.5,0.5){$3$};
\node at (1.5,0.5){$2$};
\node at (0.5,0.5){$1$};
\end{tikzpicture}
\caption{The boxed diagram of the CDTS $\mathcal{D}$ (in pink, the minimal completion curve for a 3-Dyck boxed diagram).}
\end{figure}
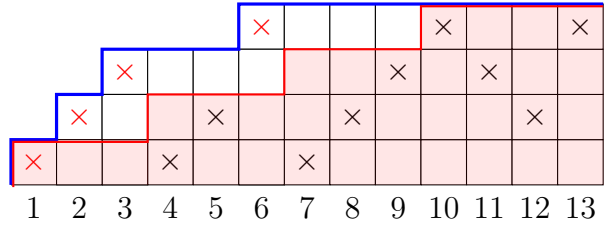

\textbf{Example:} In Figure \ref{canon}, we see how the breadth-first search of the graph produces a labeling of the vertices and edges, which, in turn, dictates the order of the search: the ends of the edges starting from a given vertex are searched in the alphabetic order, and the vertices are searched according to their order of apparition during the breadth-first search, beginning with the starting vertex.  The 24 CDTS, obtained by permutation of the symbols $\{\mathfrak{T},\mathfrak{W},\mathfrak{S},\mathfrak{X}\}$ as labels of the vertices of our example on the left, would produce the same labeling as is pictured on the left. Note that the correspondance \emph{edge-endpoint} is a surjection $\tilde{\omega}$ from $[\![1,13]\!]$ to $[\![1,4]\!]$, with a special property: the partition $P_{\mathfrak{D}}=\left(\tilde{\omega}^{-1}(1),\tilde{\omega}^{-1}(2),\tilde{\omega}^{-1}(3),\tilde{\omega}^{-1}(4)\right)$, here e.g.
$$P_{\mathfrak{D}}=\{\{1,4,7\}, \{2,5,8,12\},\{3,9,11\},\{6,10,13\}\},$$
 is necessarily sorted in increasing order of the smallest elements of the parts.

Similarly, a sequence $\omega \in \Omega_{N,n}$ (a surjection) can be matched with a boxed diagram in $\mathcal{S}_{N,n}$ in exactly $n!$ ways, as follows : according to the coupon collector metaphor, the collection process produces an order among the elements of the collection: 
$$\sigma_\omega \left( k \right) =\omega_{T_{k}\left( \omega \right)},$$
denotes the $k$th element of $[\![1,n]\!]$ to enrich the collection ; $\sigma_{\omega}$ is a random uniform permutation of $[\![1,n]\!]$. Setting
$$\tilde{\omega}_i=\sigma^{-1}_{\omega} \left(\omega_i \right),\text{~i.e.~}\tilde{\omega}=\sigma^{-1}_{\omega} \circ \omega,$$
we obtain that 
$$\tilde{\omega}_{T_k} =y_{T_k}(\omega)=k\quad\text{and}\quad \tilde{\omega}_i\le y_i(\omega)=\max\{\tilde{\omega}_j,\ j\le i\}, \forall i.$$
Thus $(y_i,\tilde{\omega}_i)_{1\le i\le N}$ is a boxed diagram associated with the surjection $\omega$, or with any surjection of the form $\tau\circ \omega$, with  $\tau\in\mathfrak{S}_n$. That is,  if $\tau\in\mathfrak{S}_n$, $\tau\circ \omega$ produces the same boxed diagram as $\omega$, and there exists exactly $n!$ elements of $\Omega_{N,n}$ with the same boxed diagram.
Finally, a random uniform surjection  $\omega \in \Omega_{N,n}$ produces a random uniform boxed diagram, while a random uniform CDTS produces a random uniform boxed diagram satisfying additionally the constraint \eqref{connexion} (such a boxed diagram is also called a $k$-Dyck boxed diagram). Thus there is a correspondance in which each boxed diagram is related to $n!$ different surjections of $\Omega_{N,n}$, and a similar (though different) correspondance in which a $k-$Dyck boxed diagram is related to $n!$ different CDTS. Note that, according to  \eqref{connexion},
\begin{align}\label{accessdeux}y_{k(n-1)+1}= \dots=y_{nk+1}=n,\end{align} 
meaning that $y_{nk+1}$ does not satisfy  inequality   \eqref{connexion}.

\subsection{Reduction to the  NorthEast corner} 
\label{NEcorner}
Now$\mathcal{A}_{k, n} $ denotes the subset of elements $\omega\in\Omega_{N,n}$ meeting the condition \eqref{connexion}. Then Theorem \ref{kor78} is equivalent to an assertion on the asymptotic  behaviour of the Markov chain $ Z^{(n)}$ studied at Section \ref{sec:rwstirling} : the probability that the sample path of $Z^{(n)}$ crosses the line $y=x /k$ outside its endpoints $(0,0)$ and $(kn+1,n)$ converges to $k\rho(k)\in(0,1)$.

\begin{figure}[ht]
\centering
\includegraphics[width=11cm]{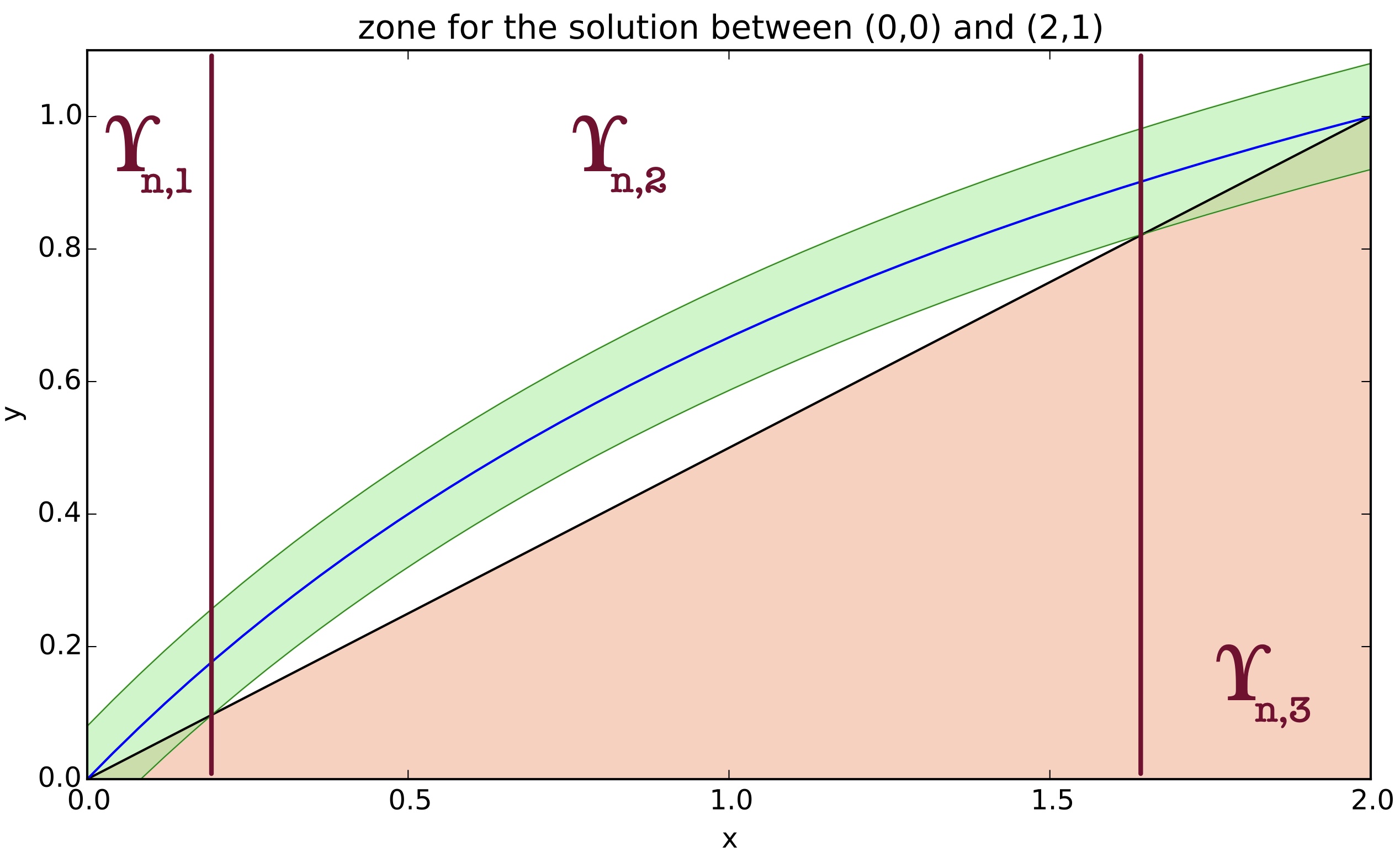}
\caption{Forbidden zones.}
\label{Saucisse}
\end{figure}

This probabilistic formulation of  Kor\v{s}unov's formula hints at the relation with Theorem \ref{mainprofile} :  as  Figure \ref{Saucisse} shows, Theorem \ref{mainprofile},  with Proposition \ref{kielbasa}, prevents these crossings outside the close vicinity of the endpoints, but for a small probability. For $n$ large enough, a crossing inside the interval $ I_2= [\![an,kn- 2 C k^2 n^{1/3} ]\!]$ violates the convergence to the limit path at the rate given by Theorem \ref{mainprofile}, thus such crossings happen with a probability smaller than $n^{1/3}e^{-\ln ^{2}n /2}$.   As a consequence, the line of proof for Theorem \ref{kor78} goes according to the following steps: set
$$I_1= [\![0, an ]\!],\quad I_3= [\![N-2Ck^{2} n^{1/3},N]\!],$$
and let the event that a crossing happens inside the interval $I_j$ be denoted $\Upsilon_{n,j}$. Then
$$\vert\mathbb{P}_{N,n} \left(\overline{\mathcal{A}_{k, n}} \right)-\mathbb{P}_{N,n} \left(\Upsilon_{n,3} \right)\vert\le \mathbb{P}_{N,n} \left(\Upsilon_{n,1} \right)+\mathbb{P}_{N,n} \left(\Upsilon_{n,2} \right),
$$
in which, due to Theorem \ref{mainprofile},
\begin{equation}
\label{crossing2}
  \mathbb{P}_{N,n} \left(\Upsilon_{n,2} \right)\le  n^{1/3}e^{-\ln ^{2}n /2}.
\end{equation}
We also have:
\begin{prop} If $a$ is small enough,
\begin{equation}
\label{crossing1}
\lim_n\mathbb{P}_{N,n} \left(\Upsilon_{n,1} \right)=0.
\end{equation}
\end{prop}
As a consequence, the asymptotic behaviour of the profile in the NorthEast corner should give simultaneously the limit of $\mathbb{P}_{N,n} \left(\Upsilon_{n,3} \right)$, and Kor\v{s}unov's formula. This is the topic of the next sections.
\begin{proof} An alternative formulation of  the condition  $\overline{\Upsilon_{n,1}}$ is as follows: $y_{\ell k+1}\geq \ell+1$ \emph{holds true for }$0\le\ell\le an$. The proof of  \eqref{crossing1}   has two steps :
\begin{equation}
\label{ramanu}
\lim_n \mathbb{P}_{n}\left(\overline{\Upsilon_{n,1}}\right)=1,
\end{equation}
and
\begin{equation}
\label{robbins}
\mathbb{P}_{N,n}\left(\overline{\Upsilon_{n,1}}\right) \ge \mathbb{P}_{n}\left(\overline{\Upsilon_{n,1}}\right).
\end{equation}
Now:
\begin{align*}
\mathbb{P}_{n}\left({\Upsilon_{n,1}}\right)
&\le 
\sum_{0\le\ell\le  an}\mathbb{P}_{n}\left( y_{\ell k+1}\leq \ell\right)
\\
&=
\sum_{1\le\ell\le  an} {n \choose \ell} \left( \dfrac {\ell}{n}\right) ^{\ell k+1} \ =\ 
\sum_{1\le\ell\le  an} \ u_{\ell}.
\end{align*}
But 
\begin{align*}
\dfrac {u_{\ell+1}}{u_{\ell}}
&= \dfrac{n - \ell}{\ell+1} \left( \dfrac {\ell+1}{\ell}\right) ^{\ell k+1} \left( \dfrac {\ell+1}{n}\right) ^{k}
\\
&\le 2\,e^k \left( \dfrac {\ell+1}{n}\right) ^{k-1} \le 2\,e^k \left(2a\right) ^{k-1}\le 4a\,e^k\le\dfrac{1}2,
\end{align*}
for $a$ small enough, in which case we have :
\begin{align*}
\mathbb{P}_{n}\left({\Upsilon_{n,1}}\right)
&\le 2 u_1\le\ \dfrac2{n^2},
\end{align*}
as expected.

For \eqref{robbins}, note that :
$$\mathbb{P}_{n}\left(\overline{\Upsilon_{n,1}}\right) = \mathbb{P}_n\left( T_{\ell+1}\leq k\ell+1,\,1+k\ell\in [\![ 0, an]\!] \right).$$ 
But, under  $\mathbb{P}_n$,  $\left(T_{k}-T_{k-1}\right)_{1\leq k\leq n}$ is a  sequence of independent random variables (with geometric distributions and respective expectations $n/(n+1-k ) $), so, according to \cite{MR0063592}:
\begin{align*}
	 \mathbb{P}_n \big( T_{\ell+1}\leq k\ell+1,\,&1+k\ell\in [\![ 0, an]\!] \ \text{and}\ T_n\le N\big)
	 \\
	 &\ge \mathbb{P}_n \left( T_{\ell+1}\leq k\ell+1,\,1+k\ell\in [\![ 0, an]\!]\right)\ \mathbb{P}_n \left(T_n\le N\right),
\end{align*} 
or, equivalently,
\begin{align*}
	 \mathbb{P}_n \big( T_{\ell+1}\leq k\ell+1,\,&1+k\ell\in [\![ 0, an]\!] \,\vert\ T_n\le N\big)
	 \\
	 &\ge \mathbb{P}_n \left( T_{\ell+1}\leq k\ell+1,\,1+k\ell\in [\![ 0, an]\!]\right),
\end{align*} 
which is relation \eqref{robbins}. 
\end{proof}

\subsection{Random walks and Pollaczek-Khintchine's formula}
\label{pollaczek} 
In this section, and the next one, we shall prove that
\begin{equation}
\label{crossing3}
\lim_n\mathbb{P}_{N,n} \left(\Upsilon_{n,3} \right)=k\rho(k).
\end{equation}
Relation \eqref{crossing3} results from  the Pollaczek-Khinchine formula, as we shall see now: $\Upsilon_{n,3}$ relates to a crossing of the line $y=x/k$ by $Z^{(n)}$ before time  $2Ck^2 n^{1/3}$. But, before time $2Ck^2 n^{1/3}$, i.e. for $(m,\ell)$ close to $(kn,n)$, due to Proposition \ref{rwstirling} and Theorem \ref{transitionp}, the transition probabilities $r(m,\ell)$ of $Z^{(n)}$ are  close to the constant $\rho(k)$, so that we expect $Z^{(n)}$ to behave, early, like a random walk $Z$  starting at $n$, with step distribution
\[(1-\rho(k))\delta_{0}+\rho(k)\delta_{-1}.\] 
Since $\rho(k)<1/k$,  the trend is that $Z$ does not  cross the line, and if it does at all, the crossing has to take place early, hence we expect the crossing probability of Z to be the limit of $\mathbb{P}_{N,n} \left(\Upsilon_{n,3} \right)$. In the next section, we shall discuss the convergence (to $Z$) of $Z^{(n)} $, and its speed. In this section, we compute the crossing probability  $\mathbb{P}\left(\Upsilon\right)$ of $Z$, in which:
$$\overline{\Upsilon} =\left\{Z_{0}=n\ \text{and}\ Z_{\ell}\ge n-\dfrac{\ell-1}{k}\ \text{for}\ \ell\ge 1 \right\}.$$ 
\begin{prop} 
\label{heurist2} 
$$\mathbb{P}\left(\Upsilon\right)=k\rho(k).$$
\end{prop}
\begin{proof}
It is convenient to make some time and space changes to represent our crossing probability in more familar terms, i.e.  in terms of a new random walk $S$ on the integers, with negative drift, starting from 0, and such that $\overline{\Upsilon} = \left\{ \max S_{n}=0\right\} $ holds true (or such that $\overline{\Upsilon}$ and $\tilde{\Upsilon}=\left\{ \max S_{n}=0\right\} $ are closely related events, actually). If we set, for $j\ge-1$,
$$S_{j}=kn-kZ_{1+j}-j, $$
then $S$ is a random walk with step distribution
\[\mu_{k}=(1-\rho(k))\delta_{-1}+\rho(k)\delta_{k-1},\] 
with drift
\[d_{k}=k\rho(k)-1<0,\]
starting from 1 at time -1, and:
$$\overline{\Upsilon}= \left\{S_{-1}=1,\,S_{0}=0\ \text{and}\ \max_{n\ge0} S_{n}=0\right\}.$$
Thus
\[\mathbb{P} \left( \overline{\Upsilon}\right) =\left( 1-\rho\left( k\right) \right) \mathbb{P} \left( \max _{n\geq 0}\left( S_{n}\right) =0\,\vert\, S_0=0\right)=\left( 1-\rho\left( k\right) \right)\mathbb{P} \left( \tilde{\Upsilon}\right) .\]
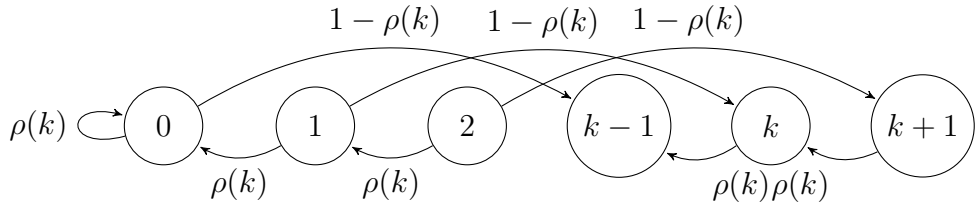
\begin{figure}[h]\centering\begin{tikzpicture}[>=stealth',shorten >=1pt,auto,node distance=2 cm, scale = 1, transform shape]
\node[state] (A)                                    {$0$};
\node[state]         (B) [right of=A]                       {$1$};
\node[state]         (C) [right of=B]                       {$2$};
\node[state] (D)        [right of=C]                            {$k-1$};
\node[state]         (E) [right of=D]                       {$k$};
\node[state]         (F) [right of=E]                       {$k+1$};

\path[->] 
(A) edge [loop left]       node [align=center]  {$ \rho(k) $} (A)
(B) edge [bend left]       node [align=center]  {$ \rho(k)$} (A)
(C) edge [bend left]       node [align=center]  {$\rho(k)$} (B) 
(F) edge [bend left]       node [align=center]  {$ \rho(k) $} (E)
(E) edge [bend left]       node [align=center]  {$\rho(k)$} (D)
(A) edge [bend left]       node [align=center]  {$1-\rho(k)$} (D)
(B) edge [bend left]       node [align=center]  {$1-\rho(k)$} (E)
(C) edge [bend left]       node [align=center]  {$1-\rho(k)$} (F);   
\end{tikzpicture}\caption{The graph of the Lindsey process.}\end{figure}
But we know, from the Pollaczek-Khinchine formula (cf. \cite[Corollary 6.6]{MR1978607}) that $\mathbb{P} \left( \tilde{\Upsilon}\right) $ is the stationary distribution $\pi_{0}$ at 0 of the Lindsey process with step $\mu_{k}$.  For $\ell\ge1$, let $t_{\ell}$ denote the average time needed by the Lindsey process (or, indifferently, by the random walk $S$) to hit 0 starting from position $\ell$: Wald's identity gives that $$t_{\ell}=-\frac\ell{d_{k}}.$$On the other hand, if $t_{0}$ is the expected time of the first return to 0, starting from 0, then, by the Markov property,
\begin{align*}\left( \dfrac {1}{\pi _{0}}=\right)\  t_{0}&=\left( 1-\rho(k)\right)+\rho(k) \left( 1+t_{k-1}\right)\\&=\left( 1-\rho(k)\right)+\rho(k) \left( 1-\dfrac {k-1}{d_{k}}\right)\\&=-\dfrac {1-\rho(k)}{d_{k}},\end{align*}
and finally:
$$\mathbb{P} \left( \tilde{\Upsilon}\right) =\pi _{0}=-\dfrac {d_{k}}{1-\rho(k)}=E_{k}.$$
Thus, as expected, $\mathbb{P} \left( \overline{\Upsilon}\right)=\left( 1-\rho\left( k\right) \right)\mathbb{P} \left( \tilde{\Upsilon}\right) =-d_k=1-k\rho(k)$. 
\end{proof} 

\subsection{Tail probabilities and Hoeffding's inequality}
\label{tails}
For some  process $X=(X_{i})_{i\ge 0}$, let $X_{[\![\ell,m]\!]}$ denote the section $(X_{i})_{\ell\le i\le m}$ of the sample path $X$. First, let us bound the distance between  the random walk $Z^{(n)}$  of Proposition \ref{rwstirling}, and  $Z$:
\begin{prop} 
\label{heurist1}
Under $\mathbb{P}_{N,n}$, $Z^{(n)}$ converges to $Z$ in distribution. Moreover, for $\alpha\in(0,1)$, there exists $C_{\alpha}>0$ such that for $n$ large enough :
\[\sup_{s\le n^{\alpha},\ A\in [\![0,n]\!]^{s+1}}\dfrac{n^{1-\alpha}}{s2^{s}}\left| \mathbb{P} \left( Z_{[\![0,s]\!]}\in A\right) -\mathbb{P}_{N,n} \left( Z^{(n)}_{[\![0,s]\!]}\in A\right) \right|\ \le\ C_{\alpha}.\]
\end{prop}
\begin{proof} 
We shall use that if
$$\theta \geq \max \left( \max \left| \alpha _{i}\right| ,\max \left| \beta _{i}\right| \right),$$ 
then:
\begin{align}
\label{glisse}
\left|\prod _{i=1}^m\alpha _{i}-\prod _{i=1}^m\beta _{i} \right|&\le \theta^{m-1}\sum_{i=1}^m \left|\alpha _{i}-\beta _{i}\right|.
\end{align}
Consider a sample path $z=\left( z_{j}\right) _{0\leq j\leq s}$ in which $z_0=n $. Let $\Delta_j= z_j-z_{j+1 }\in\{0,1\} $ denote its $j$th increment. Under $\mathbb{P}_{N,n}$, as a consequence of Proposition \ref{rwstirling}, for any given $s$,
$$\mathbb{P} \left(Z^{\left( n\right)}_{[\![0,s]\!]}=z\right)= \prod ^{s-1}_{j=0}r\left( N-j,z_{j}\right) ^{\Delta_j} \left( 1-r\left( N-j,z_{j}\right) \right) ^{1- \Delta_j}, $$  
while
\begin{eqnarray*}
\mathbb{P} \left(Z_{[\![0,s]\!]}=z\right)&= \prod ^{s-1}_{j=0} \rho\left( k\right) ^{\Delta_j} \left( 1-\rho\left( k\right) \right) ^{1- \Delta_j}.
\end{eqnarray*}
For $\alpha\in (0,1) $ and $n$ large enough, and for a suitable choice of $\eta, \delta\in (0,1) $,   $n^{-1}W_{t}$ belongs to $\coin_{\eta, 2\delta}$, so that, according to \eqref{lipsch}, for $ t= N-1 $  and $0\le \ell\le m\le n^{\alpha}$,
\begin{align}
\label{bookbook}
  \left| r\left( N-1-m,n-\ell\right) -\rho \left( k\right) \right| \leq  \dfrac {8}{\eta}\ n^{\alpha-1}.
\end{align}
Since the probability of a given sample path of $Z$, resp. $Z^{(n)}$, is a product of terms $\alpha_{i}$ or $\beta_{i}$ of the following form
$$\rho(k)^{\Delta}(1-\rho(k))^{1-\Delta},\quad\text{resp.}\quad r(m,\ell)^{\Delta}(1-r(m,\ell))^{1-\Delta},$$
in which $-\Delta\in\{0,-1\}$ is the increment for some step of the random walks of $Z$, resp. $Z^{(n)}$, and as a consequence of  \eqref{bookbook} and \eqref{glisse} (with $\theta=1$), the probabilities of these sample paths of length $s\le n^{\alpha}$ differ by at most
$$\dfrac {8}{\eta }\ sn^{\alpha-1}.$$
That a set $A\subset [\![0,n]\!]^{s+1}  $ has at most $2^{s}$ admissible elements starting at position $n$ entails that Proposition \ref{heurist1} holds true with the choice $ C_{\alpha}= \dfrac {8}{\eta }$.
\end{proof}
Let $\mathbb{Z}^{\star}$ (resp. $\mathbb{Z}^{\infty}$)  denote the set of finite (resp. finite or infinite) words on the alphabet $\mathbb{Z}$, and for a finite word $\omega=\omega_0\omega_1\omega_2\dots\omega_s$, set $|\omega|=s+1$. For $0\le s<t\le +\infty$, let us define the crossing set $\Upsilon(s,t)$ as follows:
$$\Upsilon(s,t)= \left\{\omega\in\mathbb{Z}^{\infty}\ \text{s.t.}|\omega|\ge t+1\text{~and~}\ \exists \ell\in [\![s,t]\!]\text{~s.t.~}\omega_{\ell}< n-\dfrac{\ell-1}{k}\right\},$$ 
so that, for instance, 
$$\mathbb{P}\left(\Upsilon_{n,3}\right)= \mathbb{P}\left(Z^{(n)}\in \Upsilon(1,2k^2n^{1/3})\right).$$
The next proposition completes the proof of Theorem \ref{kor78} :
\begin{prop} 
\label{heurist2}  $$\lim_n\mathbb{P}\left(\Upsilon_{n,3}\right) =k\rho(k).$$
\end{prop}
\begin{proof} We shall prove successively that:
\begin{align}
\label{chain1}
 \mathbb{P}\left(\Upsilon\right)
 &= \lim_n \mathbb{P}\left(Z\in\Upsilon(1,s_n)\right)
\\
\label{chain2}
&= \lim_n \mathbb{P}\left(Z^{(n)}\in\Upsilon(1,s_n)\right)
\\
\label{chain3}
&= \lim_n \mathbb{P}\left(\Upsilon_{n,3}\right),
\end{align}
for $s_n=\upsilon\ln n$, in which  $\upsilon\ln 2\le 1-\alpha$. First, Proposition \ref{heurist1} entails \eqref{chain2} at once. Relations \eqref{chain1} and  \eqref{chain3} 
both follow from Hoeffding's inequality.   For \eqref{chain1} it is rather straightforward : if we set $\beta(k)=\rho(k)-\tfrac{1}{k}$, relation \eqref{azuma} entails that $\beta(k)<0$, so that
\begin{align}
\label{azu}
	\mathbb{P} \left(Z_{\ell}-Z_0\le-\dfrac{\ell}{k}\right)=\mathbb{P} \left(Z_{\ell}-Z_0+\ell\rho(k)\le \beta(k)\ell\right)\le \exp\left(-2\beta(k)^2\ell\right).
\end{align}
Thus the probability of a crossing at some point after time $s_{n}$ satisfies
\begin{align*}
\mathbb{P}\left(\Upsilon\right)-\mathbb{P}\left(Z\in\Upsilon(1,s_n)\right)
&\le\mathbb{P}\left(Z\in\Upsilon(s_n,+\infty)\right)
\\
&\le\sum_{\ell\ge s_{n}}\mathbb{P} \left(Z_{\ell}-Z_0\le-\dfrac{\ell}{k}\right)
\\
&\le \sum_{\ell\ge s_{n}}\exp\left(-2\beta(k)^2\ell\right)=\dfrac{n^{-2\upsilon\beta(k)^2}}{1-e^{-2\beta(k)^2}}\ .
\end{align*}
Similarly
\begin{align}
\label{bookbook2}
\mathbb{P}\left(Z^{(n)}\in\Upsilon(1,s_n)\right)-\mathbb{P}\left(\Upsilon_{n,3}\right)
&\le\sum_{s_{n}\le \ell\le 2k^2n^{1/3} }\mathbb{P} \left(Z^{(n)}_{\ell}-Z^{(n)}_0\le-\dfrac{\ell}{k}\right),
\end{align}
but here we cannot use Hoeffding's inequality directly, though $Z^{(n)}_{0}-Z^{(n)}_{\ell }$ is  a sum of Bernoulli random variables, for these  Bernoulli random variables  are not independent. However, we can build, on the same probability space, a copy of $Z^{(n)}_{ }$ and a random walk $\hat{Z}$ in such a way that, for $\ell\le 2k^2n^{1/3} $,  $Z^{(n)}_{m }\le\hat{Z}_{m}$ and $\hat{Z}$'s drift is smaller than $1/k$, using a sequence $U=\left( U_{m}\right) _{m\ge 1}$ of independent random variables, uniform on $(0,1)$. Set $b= \dfrac {1}{k}+\dfrac {\beta \left( k\right) }{2} $ and 
$$\widehat {Z}_{m}-\widehat {Z}_{m+1}=1_{U_{m}\leq b}.$$
For $\widehat {Z}^{(n)}$ and $ {Z}^{(n)}$ to have the same distribution, due to Proposition \ref{rwstirling} , we need to set $\widehat {Z}^{(n)}_0= {Z}^{(n)}_0=n$ and 
$$\widehat {Z}^{(n)}_{m}-\widehat {Z} ^{(n)}_{m+1}=1_{U_{m}\leq r \left(N-m, \widehat {Z}^{(n)}_{m} \right) }.$$
For   $0\le n- \widehat {Z}^{(n)}_{m}\le m\le n^{\alpha}$,  and for $\alpha> 1/3 $, if we choose $n $ large enough, so that $n^{\alpha}>2k^2n^{1/3}$, and so that
$$\dfrac {8}{\eta}\ n^{\alpha-1}\le -\dfrac {\beta \left( k\right) }{2},$$
we can use \eqref{bookbook} to obtain that
$$\forall m\in [\![0, 2k^2n^{1/3}]\!] ,\qquad r \left(N-m, \widehat {Z}^{(n)}_{m} \right)  \le b,\text{~and~} \widehat {Z}^{(n)}_{m} \ge \widehat {Z}_{m}.$$
Hence
$$\mathbb{P} \left(Z^{(n)}_{m}-Z^{(n)}_0\le-\dfrac{m}{k}\right) \le\mathbb{P} \left(\widehat{Z}_{m}-\widehat{Z}_0\le-\dfrac{m}{k}\right) \le \exp\left(-\beta(k)^2 m/2\right),$$
the second inequality due to Hoeffding's inequality. Relation \eqref{chain3} follows.
\end{proof}

\subsection{The profile of an accessible automaton}
\label{automataprofile}
Jointly with Theorem \ref{kor78}, Theorem \ref{mainprofile} has a straightforward consequence : the completion curve of a uniform ADCA has the same limit curve.  More precisely, if $\mathbb{Q}_{k, n}$ denotes the uniform distribution on ADCA with $n$ vertices and $k$ letters, or, equivalently $\mathbb{Q}_{k, n}$  is the conditional probability given  $\mathcal{A}_{k, n}$:
$$\mathbb{Q}_{k,n}\left( B\right) =\dfrac {\mathbb{P}_{N,n} \left( B\cap \mathcal{A}_{k, n}\right) }{\mathbb{P}_{N,n} \left(\mathcal{A}_{k, n} \right) }$$
then
\begin{lem}
For a sequence of events $\left( B_{n}\right)_{n\ge n_0}$, 
$$\mathbb{Q}_{k,n} \left( B_{n}\right) = \mathcal{O}\left(\mathbb{P}_{N,n}\left( B_{n}\right)\right).$$
\end{lem}
Thus, Theorem \ref{mainprofile} translates to large automata at once, and we obtain
\begin{theo}\label{autoprofile} For any $a  >0$, there exists $C_{3 }(n_{0},\varepsilon)>0$ such that, for $n\ge n_{0}$,
\begin{equation*}
\mathbb{Q}_{k,n} \left( \sup_{[\varepsilon,k]} \left| \zeta    _{n}-f_{k-1 } \right| \geq Cn^{-1 /3}\right) \leq C_{4}n^{1/3}e^{-\ln ^{2}n /2}.
\end{equation*}
\end{theo}
Recall that $f _\Lambda$ is defined at the end of Section \ref{sec:coupon}.

\section{Saddle-point method and Stirling numbers}
\label{saddle}
This section is devoted to the proof of Theorem \ref{goodplus}.
\subsection{Generating function and Cauchy formula}
Recall the notations:
\begin{align}
\lambda(m,\;\ell )= \lambda=\dfrac{m-\ell}{\ell},\quad \psi(m,\ell)=\frac{1}{2\pi}\ \frac{m!}{\ell !}\ \left(\frac{e^{\impl  }-1}{\impl  ^{1+\lambda}}\right)^{\ell }\ \sqrt{\frac{\pi}{ v\ell}}.
\end{align}
According to \cite[(6)]{MR0120204} or \cite[Example III.11, p.179]{MR2483235}, we have :
\begin{align*}
\sum\limits_{m\geq1}{m\brace \ell }\frac{z^m}{m!}&=(e^z-1)^\ell \frac{1}{\ell !},\\
&=\tfrac{z^\ell }{\ell !}B(z)^\ell,
\end{align*}
in which
\[B(z):=\dfrac{e^z-1}{z}.\]
By the Cauchy formula,
\begin{align*}
{m \brace \ell }&=[z^{m - \ell  }]\frac{m!}{\ell !}B(z)^\ell=\frac{1}{2i\pi}\oint\frac{m!}{\ell !}B(z)^\ell \frac{dz}{z^{m - \ell  +1}},\\
&=\frac{1}{2i\pi}\int_{-\pi}^{\pi}\frac{m!}{\ell !}B(\impl   e^{i\theta})^\ell (\impl   e^{i\theta})^{-m +\ell  -1}\impl   ie^{i\theta}d\theta,\\
&=\frac{1}{2\pi}\frac{m !}{\ell !}\left(\frac{B\left(\impl  \right) }{\impl  ^{\lambda  }}\right)^\ell\int_{-\pi}^{\pi}g\left(\theta\right)^\ell d\theta\\
&=a_{\ell}\ \int_{-\pi}^{\pi}g\left(\theta\right)^\ell d\theta.
\end{align*}
in which 
\begin{align*}
B(\impl   e^{i\theta})e^{-i\lambda\theta}
&=B\left( \impl   \right)g(\theta),
\end{align*} 
to be compared to  the asymptotic equivalent  to ${m\brace \ell}$ given by \cite[(3)]{MR0120204}, $\psi(m,\ell)$, that satisfies:
\begin{align*}
\psi(m,\ell)&=a_\ell\ \sqrt{\frac{\pi}{ v\ell}}\ =\ a_{\ell }\,\int_{-\infty}^{\infty}e^{-\ell v\theta^2}d\theta,\\
\end{align*} 
We expect that $|g(\theta)|\le 1$ for any $\theta$, or $|B(\impl   e^{i\theta})|\le B(\impl  )$,  since $B(\impl   z)$, as a power series in $z$, has positive coefficients. We also expect that, around 0,
$$g(\theta)\simeq_{0} 1,$$
and more precisely, since $\impl  $ is a saddle-point, we expect that
$$g(\theta)= e^{-v\theta^{2}+\mathcal{O}(\theta^{3})}.$$
According to \eqref{ineqed1}, 
\begin{align*}
v&=\dfrac{(\lambda+1)(\impl  -\lambda)}2\ge \dfrac{\lambda}2>0,
\end{align*}
which entails that 
$$\int_{-\pi}^{\pi}g\left(\theta\right)^\ell d\theta\simeq \int_{-\infty}^{+\infty} e^{-v\theta^{2}\ell} d\theta=\sqrt{\dfrac\pi{v\ell}}.$$
Set
\begin{align}
\label{kazero}
K_\ell ^{(0)}&=\left|\sqrt{\dfrac\pi{v\ell}}-\int_{-\theta_0}^{\theta_0}g\left(\theta\right)^\ell d\theta\right|,
\\
\label{kaun}
K_\ell ^{(1)}&=2\left|\int_{\theta_0}^{\pi}g\left(\theta\right)^\ell  d\theta\right|,
\end{align}
in which a suitable choice of $\theta_{0}$ is made later, so that:
\begin{align}
\nonumber
\psi(m,\ell)&=a_\ell\ \sqrt{\frac{\pi}{ v\ell}}\ =\ a_{\ell }\,\int_{-\infty}^{\infty}e^{-\ell v\theta^2}d\theta,
\\
|\chi(m,\ell)|=\dfrac{\left|\psi(m,\ell)-{m\brace \ell}\right|}{\psi(m,\ell)}&\le\ \sqrt{\frac{ v\ell}{\pi}}\ \left(K_\ell ^{(0)}+K_\ell ^{(1)}\right).
\end{align}
In the next sections, in order to prove Theorem \ref{goodplus},  we  obtain that $$K_\ell ^{(0)}+K_\ell ^{(1)}=\mathcal{O}\left( \ell^{-3/2}\right).$$ 
\subsection{Central term}
In this  section we obtain a  saddlepoint bound for $K_\ell ^{(0)}$, following  \cite{MR2483235}. We write 
\begin{align}
\nonumber
g(\theta)
&=e^{-i\left( 1+\lambda \right) \theta }\quad\dfrac {e^{\impl   \left( e^{i\theta }-1\right) }-e^{-\impl  }}{1-e^{-\impl  }}
\\
\label{poissoncar}
&=e^{-i\left( 1+\lambda \right) \theta }\quad\dfrac {\Phi(\theta)-e^{-\impl  }}{1-e^{-\impl  }}
\end{align} 
in which :
$$\Phi(\theta)=e^{\impl   \left( e^{i\theta }-1\right) }=1+i\impl   \theta -\dfrac {\impl   ^{2}\theta ^{2}}{2}+\mathcal{O}\left( \theta ^{3}\right),$$
is the characteristic function of any Poisson  random variable $Z$ with expectation $\impl$. For our aims,  we need a precise estimation of $g$,  obtained through the Taylor-Laplace inequality, see Section \ref{sec:taylor}. There, we prove that, for suitable constants $(v,\tau,\gamma)$,
\begin{align}
\label{affine6}
\left| g\left( \theta \right) -\left( 1-v\theta ^{2}+\tau \theta ^{3}+\gamma \theta ^{4}\right)\right| \leq T\left(\lambda \right) \theta^{5},
\end{align}
in which, according to   Section \ref{beeeurk},
$$T\left(\lambda \right) =\dfrac {\left( 1+\lambda \right)^{6}}{2\lambda}.$$
Note that, according to \eqref{ineqed1}, 
\begin{align*}
v&=\dfrac{(\lambda+1)(\impl  -\lambda)}2\ge \dfrac{\lambda}2>0.
\end{align*}
We can write
\begin{align*}
K_{\ell}^{(0)}&\le \int_{-\theta_0}^{\theta_0}\left|g(\theta)^{\ell }-e^{-\ell v\theta^2}\right|d\theta+2\,\int_{\theta_0}^{\infty}e^{-\ell v\theta^2}d\theta\\
&=K_{\ell}^{(00)}+K_{\ell}^{(01)}.
\end{align*}
Now
\begin{align*}
K_{\ell}^{(01)}&=  \dfrac{2}{\sqrt{v\ell}}\,\int_{\theta_0\sqrt{v\ell}}^{\infty}e^{-x^2}dx
\\
&\le  \dfrac{2}{\sqrt{v\ell}}\,\int_{\theta_0\sqrt{v\ell}}^{\infty}\ \dfrac{2xe^{-x^2}}{2\theta_0\sqrt{v\ell}}\ dx
\\
&=\ \dfrac{1}{\theta_0 v\ell}\ e^{-\theta_0^2v\ell}.
\end{align*}
Thus,  $\theta_0\sqrt{ \ell}$ has to be large for $K_{\ell}^{(01)}$ to be $o\left(\ell^{-3/2}\right)$. On the other hand, 
\begin{align*}
K_{\ell}^{(00)}&\le K_{\ell}^{(000)}+K_{\ell}^{(001)}+K_{\ell}^{(002)},
\end{align*}
in which, for $\tilde{\gamma}=\gamma-\tfrac{v^{2}}2$, 
\begin{align*}
K_{\ell}^{(000)}&=\int_{-\theta_0}^{\theta_0}\left|g(\theta)^{\ell }-\left(1-v\theta^2+\tau\theta^3+\gamma\theta^4\right)^{\ell }\right|d\theta,
\\
K_{\ell}^{(001)}&=\int_{-\theta_0}^{\theta_0}\left|\left(1-v\theta^2+\tau\theta^3+
\gamma\theta^4\right)^{\ell }-e^{-\ell(v\theta^2-\tau\theta^3-\tilde{\gamma}\theta^4)}\right|d\theta,
\\
K_{\ell}^{(002)}
&=\,\left|\int_{-\theta_0}^{\theta_0}\left(e^{-\ell(v\theta^2-\tau\theta^3-\tilde{\gamma}\theta^4)}-e^{-\ell v\theta^2}\right)d\theta\right|.
\end{align*}
With the help of \eqref{affine6},
since $T(\lambda)$ is bounded for $\lambda\in(\delta,\delta^{-1})$, we obtain that
\begin{align*}
K_{\ell}^{(000)}&\le \ell\,\int_{-\theta_0}^{\theta_0}\left|g(\theta)-\left(1-v\theta^2+\tau\theta^3+\gamma\theta^4\right)\right|d\theta
\\
&\le \,C_{\texttt{000}}\ \ell\ \theta_0^{6},
\end{align*}
in which $C_{\texttt{000}}$ is discussed at Section \ref{beeeurk}. For $K_{\ell}^{(000)}$ to be small, $\theta_0\sqrt{ \ell}$ cannot be too large: 
$$\theta_{0}=\dfrac{\ln\ell}{\sqrt\ell}$$
yields that
\begin{align*}
K_{\ell}^{(000)}&\le C_{\texttt{000}}\,\dfrac{\ln^{6}\ell}{\ell^{2}}=o\left(\ell^{-3/2}\right),
\end{align*}
and that, for $\lambda\in(\delta,\delta^{-1})$, and $\ell\ge e^{3/\delta}$,
\begin{align*}
K_{\ell}^{(01)}&\le\ \dfrac{1}{ v\ln\ell}\ \ell^{-1/2-v\ln\ell}
\\
&\le\ \dfrac{2}{ \delta\ln\ell}\ \ell^{-(1+\delta\ln\ell)/2}\ \le \dfrac1{\ell^{2}}.
\end{align*}
Now
\begin{align*}
K_{\ell}^{(001)}&\le \ell\,\int_{-\theta_0}^{\theta_0}\left|1-v\theta^2+\tau\theta^3+\gamma\theta^4-e^{-v\theta^2+\tau\theta^3+\tilde{\gamma}\theta^4}\right|d\theta
\\
&\le C_{\texttt{001}}\,\dfrac{\ln^{6}\ell}{\ell^{2}}=o\left(\ell^{-3/2}\right),
\end{align*}
in which the dependence of $(C_{\texttt{000}},C_{\texttt{001}},\tau,\gamma,\tilde{\gamma}) $ on $\lambda$ is studied at Section \ref{sec:taylor}, in order to complete the proof of Theorem \ref{goodplus}.  Finally
\begin{align}
\nonumber
K_{\ell}^{(002)}
&\le\,\ell\,e^{\ell|\tilde{\gamma}|\theta_{0}^4}\,
\left|\int_{-\theta_0}^{\theta_0}e^{-\ell v\theta^2}
\left(e^{\tau\theta^3+\tilde{\gamma}\theta^4}-1\right)d\theta\right|,
\\
\label{002}
&\le\,2\ell\,
\left|\int_{-\theta_0}^{\theta_0}e^{-\ell v\theta^2}
\left(e^{\tau\theta^3+\tilde{\gamma}\theta^4}-1-\tau\theta^3-\tilde{\gamma}\theta^4\right)d\theta\right|
\\
\nonumber
&\hspace{3cm}
+\,2\ell \,\left|\int_{-\theta_0}^{\theta_0}e^{-\ell v\theta^2}\left(\tau\theta^3+\tilde{\gamma}\theta^4\right)d\theta\right|
\\
\nonumber
&= K_{\ell}^{(0020)}+ K_{\ell}^{(0021)}.
\end{align}
For  inequality \eqref{002}, note that, due to inequalities \eqref{ineqcoeffs}, if $\lambda(m,\ell)\in(\delta,\tfrac1\delta)$, then $\ell|\tilde{\gamma}|\theta_0^4\le \ln2$ for $\ell$ large enough.  For the first term, since
$$\left| e^{z}-1-z\right| \leq \dfrac {\left| z\right| ^{2}}{2}\sup_{u\in[0,1]} \left| e^{uz}\right|, $$
and $\tau\in i\R$, we have
\begin{align}
\nonumber
K_{\ell}^{(0020)}
&\le\,\ell\  e^{|\tilde{\gamma}|\theta_0^4}  \int_{-\theta_0}^{\theta_0}e^{-\ell v\theta^2}\left(-\tau^2\theta^6+\tilde{\gamma}^2\theta^8\right)d\theta
\\
\nonumber
&\le\,\ell\  e^{|\tilde{\gamma}|\theta_0^4} \left(-\tau^2+\tilde{\gamma}^2\right) \int_{-\theta_0}^{\theta_0}e^{-\ell v\theta^2}\theta^6d\theta
\\
\nonumber
&=\,\ell\,(\ell v)^{-7/2}\  e^{|\tilde{\gamma}|\theta_0^4} \left(-\tau^2+\tilde{\gamma}^2\right) \int_{-\sqrt{v}\ln\ell}^{\sqrt{v}\ln\ell}e^{-x^2}x^6dx
\\
\nonumber
&\le\,\ell^{-5/2}\,v^{-7/2}\  e^{|\tilde{\gamma}|\theta_0^4} \left(-\tau^2+\tilde{\gamma}^2\right)\Gamma\left(7/2\right)
\\
\label{0020}
&\le\,2\sqrt{\pi}\,v^{-7/2}\left(-\tau^2+\tilde{\gamma}^2\right)\   \ell^{-5/2}
\\
\nonumber
&\le\  C_{\texttt{0020}}\ell^{-5/2},
\end{align}
for $\ell$ large enough. Also:
\begin{align*}
K_{\ell}^{(0021)}
&=\,2\ell\,\left|\int_{-\theta_0}^{\theta_0}e^{-\ell v\theta^2}\tilde{\gamma}\theta^4\ d\theta\right|,
\\
&\le\,|\tilde{\gamma}|\ell^{-3/2}v^{-5/2}\int_{\R}e^{-x^2}x^4\ dx,
\\
&\le\,\sqrt\pi|\tilde{\gamma}|v^{-5/2}\ell^{-3/2}\le\, C_{\texttt{0021}}\ell^{-3/2}.
\end{align*}
Thus
\begin{align}
\nonumber
\dfrac{\left|a_\ell K_{\ell}^{(0)}-\psi(m,\;\ell ))\right|}{ \psi(m,\ell) }
&\le
\left((C_{\texttt{000}}+C_{\texttt{001}})\dfrac{\ln^{6}\ell}{\ell^{2}}+C_{\texttt{0020}}\ell^{-5/2}+C_{\texttt{0021}}\ell^{-3/2}+\dfrac{1}{\ell^{2}}
\right)\,\sqrt{\dfrac{v\ell}\pi}
\\
\label{affine1}
&=C_{\texttt{0021}}\sqrt{\tfrac{v}{\pi}}\ \dfrac{1}{\ell}+\mathcal{O}\left(\dfrac{\ln^{6}\ell}{\ell^{3/2}}\right).
\end{align}
\subsection{Tail}
As for $K_\ell ^{(1)}$, relation \eqref{kaun} yields that:
\begin{align}
\nonumber
\left|K_\ell ^{(1)}\right|&\le 2\int_{\theta_0}^{\pi}\left|g\left(\theta\right)\right|^\ell d\theta.
\end{align}
Set:
$$h \left( x\right) =\dfrac {1}{\pi ^{2}}\dfrac {2x^{2}}{\left( 2+x\right) \left( e^{x}-1\right) }.$$
Following \cite[Lemma 1\& 2]{MR1078714}, we prove that
\begin{lem}\label{MeirMoon}
For $\theta\in[-\pi,\pi]$,
$$\left| B\left( \impl   e^{i\theta }\right) \right| \le B\left( \impl   \right) e^{-h \left( \impl  \right) \theta ^{2}},$$
or equivalently
$$\left| g\left(\theta \right) \right| \le e^{-h \left( \impl  \right) \theta ^{2}}.$$
\end{lem} 
\begin{proof}For $k\ge0$, set:
$$b_{k}=\dfrac {\impl   ^{k}}{k+1!},
$$
so that:
$$ B\left( \impl   e^{i\theta }\right) =\sum _{k\geq 0}b_{k}e^{ik\theta },$$
and:
$$\left| B\left( \impl   e^{i\theta }\right) \right| \le \left|b_0+b_1e^{i\theta }\right| +\sum_{k\geq 2}b_{k}.$$
But
\begin{align*}
\left| b_{0}+b_{1}e^{i\theta }\right| ^{2}&=\left( b_{0}+b_{1}\cos \theta \right) ^{2}+b_{1}2\sin ^{2}\theta 
\\
&=\left( b_{0}+b_{1}\right) ^{2}+2b_{0}b_{1}\left( \cos \theta -1\right)
\\
&=\left( b_{0}+b_{1}\right) ^{2}-4b_{0}b_{1}\sin^2\left( \tfrac\theta2\right)
\\
&\le \left( b_{0}+b_{1}-\tfrac{2b_{0}b_{1}\sin^2\left( \tfrac\theta2\right)}{b_{0}+b_{1}}\right) ^{2}.
\end{align*}
Thus
$$\left| B\left( \impl   e^{i\theta }\right) \right| \le \sum_{k\geq 0}b_{k}-\tfrac{2b_{0}b_{1}\sin^2\left( \tfrac\theta2\right)}{b_{0}+b_{1}}=B\left( \impl   \right) -\dfrac {2\impl   }{2+\impl   }\sin ^{2}\left( \dfrac {\theta }{2}\right) .$$
For $\theta\in[-\pi,\pi]$,
$$\sin ^{2}\left( \dfrac {\theta }{2}\right) \geq \dfrac {\theta ^{2}}{\pi ^{2}},$$
leading to:
$$\left| B\left( \impl   e^{i\theta} \right) \right| \le B\left( \impl   \right) \left( 1-h\left( \impl   \right) \theta ^{2}\right) \le B\left( \impl   \right) e^{-h\left( \impl   \right) \theta ^{2}}.$$
\end{proof}

Thus, according to Section \ref{lastbound}, for $\ell$ large enough,
\begin{align}
\label{affine2}
\left|K_\ell ^{(1)}\right|&\le 2\pi  \ell^{-h \left( \impl  \right) \ln(\ell)}=o\left( \ell^{-3/2}\right).
\end{align}
Finally we are ready to prove Theorem \ref{goodplus}.
\begin{proof}
Inequality  \eqref{affine1} holds true for $\ell $ large enough, and, on $(0,+\infty)$, its coefficients $C_i$  are  positive continuous functions of $\lambda$, thus, for $\lambda(m,\ell)\in[\delta,\tfrac1\delta]$, they are bounded. One can deal similarly with inequality  \eqref{affine2} (see  Section \ref{lastbound}).
\end{proof}

\bibliographystyle{amsalpha}
\bibliography{coupon}

\section{Appendix: Some special functions}
\label{endix}  

\subsection{$\impl  $ as an implicit function of $\lambda$}
We have seen that  $\impl  $ is an essential parameter in the asymptotic behaviour of the Stirling number 
$${m\brace \ell}={\left( 1+\lambda \right) \ell \brace \ell},$$
in which $\lambda= \lambda(m,\ell)$ is defined by $\lambda=\tfrac{m-\ell}\ell$, and $\impl  $  is  an implicit function of $\lambda$, defined by
\begin{align*}
\left( 1+\lambda \right)\left( 1 -e^{-\impl  }\right) &= \impl  ,\quad \lambda, \impl  \ge0.
\end{align*}
For instance,  the completion curve $\zeta_{\Lambda}$ of Theorem \ref{mainprofile} is defined in terms of $\Lambda=\lambda(N,n)$ and in terms of $\Xi=\impl  (\Lambda)$.  Thus we need to list some of the properties of $\impl  $ that are of interest in our proofs, not all of them being straightforward, for instance in order to prove Theorem \ref{transitionp} in Section \ref{sec:transition}.
\begin{prop}
The function $\impl$ is increasing, nonnegative and concave, and $ \lambda \rightarrow\impl(\lambda )- \lambda $ is  increasing, nonnegative and concave as well. Also, we have:
\begin{align}
\label{zetabound1}
\lambda \leq \impl \left( \lambda \right) &\leq \min \left( 2\lambda ,1+\lambda \right), \quad \lambda \ge0,
\\
\label{symptote}
\impl(\lambda )&= 1 + \lambda  +\mathcal{ O}_{ +\infty}\left( \lambda  e^{-\lambda }\right),
\\
\label{azuma}
e^{-\impl}&<\dfrac1{1+\lambda}.
\end{align}
\end{prop}
\begin{proof}

{\bf Proof of \eqref{zetabound1}.}
Relation \eqref{zeta} entails
$$\impl\le 1+ \lambda $$
 at once. Since $\impl\ge0$,
$$\left\{ 1+\impl \geq \dfrac {\impl }{1-e^{-\impl }}\right\} \Leftrightarrow \left\{ e^{\impl }\geq 1+\impl \right\}, $$
so, from
$$\dfrac {\impl }{1-e^{-\impl }}=1+\lambda ,$$
we deduce that
$$\impl\ge \lambda .$$
In order to prove that $2\lambda \ge\impl$, we need to prove that
\[1+\lambda =\dfrac {\impl }{1-e^{-\impl }}\ge 1+\dfrac {\impl }{2},\]
but the last inequality holds true for any positive number $\impl$, as a consequence of 
$$e^{-2x}\geq \dfrac {1-x}{1+x},\quad  x\geq 0.$$

{\bf Proof of monotony and concavity of $\impl$ and $\impl-\lambda$.} Note that
\[e^{-\impl(0) } =1-\impl(0)\]
 entails $\impl(0)=0$. For $\impl '\left( 0\right)=2$ we have no additional trouble: when $\impl, \lambda \rightarrow 0_+$,
\[1+\lambda =\dfrac {\impl }{1-e^{-\impl }}= 1+\dfrac {\impl }{2}+ o\left(\impl^2\right).\]
Now, from the implicit function theorem, we obtain:
\begin{align}
\nonumber\impl'\left( \lambda \right)&=-\dfrac {e^{-\impl }-1}{1-\left( 1+\lambda \right) e^{-\impl }}
\\
\label{derive}
&=\dfrac {\impl}{ (1+\lambda )(\impl -\lambda )}
\\
\nonumber&=1+\dfrac {\lambda e^{-\impl }}{\impl -\lambda }
\end{align}
entailing that $\impl$, and $\impl-\lambda $ as well, are increasing. Then
\begin{align*}
\impl ''\left( \lambda \right)&=\dfrac {\impl ' (1+\lambda )(\impl -\lambda )-\impl(\impl -\lambda )-\impl(1+\lambda )(\impl' -1)}{ (1+\lambda )^2(\impl -\lambda )^2}
\\
&=\dfrac {\impl(1-\impl +\lambda -(1+\lambda )\impl' +1+\lambda )}{ (1+\lambda )^2(\impl -\lambda )^2}
\\
&=\dfrac {\impl((2+2\lambda -\impl)(\impl -\lambda ) -\impl )}{ (1+\lambda )^2(\impl -\lambda )^3}
\end{align*}
so that
\begin{align*}
\impl ''\left( \lambda \right)&=\dfrac {-\impl(\impl-2\lambda )(\impl-\lambda -1)}{ (1+\lambda )^2(\impl -\lambda )^3}\le 0.
\end{align*}
It follows that $\impl$ and $\impl- \lambda $ are concave. Finally, \eqref{symptote} is an easy consequence of \eqref{zeta}, and \eqref{azuma} follows from:
$$e^{-\impl   }\left( 1+\lambda \right) =1 +\lambda -\impl  ,$$
and from \eqref{zetabound1}.\end{proof}
We also  need that :
\begin{lem}
\begin{equation}
\label{ineqed1}
2v\ge \lambda. \end{equation}
\end{lem}
\begin{proof}
The relation \eqref{ineqed1} can be written successively:
\begin{align*}
\left (\impl  -\dfrac {\impl   }{1-e^{-\impl   }}+1 \right) \left( \dfrac {\impl   }{1-e^{-\impl   }} \right)&\ge \dfrac {\impl   }{1-e^{-\impl   }}-1,
\\
\dfrac {\impl  ^2 e^{-\impl   }}{ \left( 1-e^{-\impl   } \right)^2}&\le 1,
\\
\impl   ^{2}&\le 4\sinh^2 \left( \dfrac {\impl   }{2}\right) = 2\cosh \left( \impl   \right) -2,
\\
1+\dfrac {\impl   ^{2}}{2}&\le \cosh \left( \impl   \right) ,
\end{align*}  	
the last one being clearly true.
\end{proof}

\subsection{Large deviation for the coupon collector}
\label{largedev}
Since, for $\lambda  > 0$, we have:
\begin{align*}
\mathbb{P}\left( T_{n}\left( \omega \right) \leq \left\lfloor\left( 1+\lambda \right) {n}\right\rfloor \right)&=n!\left\{\begin{array}{c}\left\lfloor\left( 1+\lambda \right) {n}\right\rfloor) \\ n \end{array}\right\}n^{-\left\lfloor\left( 1+\lambda \right) {n}\right\rfloor},
\end{align*}
Theorem \ref{good} entails that
$$\mathbb{P}\left( T_{n}\left( \omega \right) \leq \left\lfloor\left( 1+\lambda \right) {n}\right\rfloor \right)\simeq\sqrt {\tfrac {e^{\impl   }-1}{e^{\impl   }-1-\impl   }}\ e^{-nJ \left( \impl   \right) },
$$
in which \begin{align*}
J \left( \impl   \right)&= \dfrac {\impl   }{1-e^{-\impl   }}\left( 1 -\impl   +\ln \left( e^{\impl   }-1\right)\right) -\ln \left( e^{\impl   }-1\right).
\end{align*}
One can write :
\begin{align*}
J \left( \impl   \right)&= \dfrac {(\impl  -1+e^{-\impl   })\ln \left( e^{\impl   }-1\right)+ \impl  (1-\impl  )}{1-e^{-\impl   }}
\end{align*}
and finally
\begin{align*}
J '\left( \impl   \right)\left( e^{\impl  }-1\right) ^{2}&= e^{\impl  }\left( e^{\impl  }-1-\impl  \right) \ln \left( 1-e^{-\impl  }\right) <0.
\end{align*}
Also:
\begin{align*}
\left( 1-e^{-\impl  }\right) J \left( \impl   \right)&=(\impl  -1+e^{-\impl   })\left(\impl   +\ln \left(1-e^{-\impl   }\right)\right)+ \impl  -\impl  ^2
\\
&=(\impl  -1+e^{-\impl   })\ln \left(1-e^{-\impl   }\right)+\impl   e^{-\impl   }=\mathcal{O}_{+\infty} \left( \impl   e^{-\impl  }\right) 
\end{align*}
Thus, $J$ is decreasing and
$$\lim _{+\infty }J \left( \impl   \right) =0,$$
which entails that $J$ is positive.

\subsection{Properties of the limit path.} \label{limitpath} The properties of the sample path $\zeta   =f_{\lambda(x_0,y_0)}$ solution of
\begin{equation}
\label{edfond}
y'=e^{-\impl   \left( \tfrac {x}{y}-1\right) },\quad (x,y)\in A=\left\{ 0 < y < x\right\},\text{~ and~} y(x_0) = y_0,
\end{equation}
between $(0,0)$ and $(x_0,y_0)$ matter to our saddle-point estimates for the Stirling numbers, since these estimates are valid only when the sample path is far away from $\partial A$, or, equivalently, when  $x$ is large enough and $\lambda$  is far  from $\{0,+\infty\}$. We are specially interested by the solution $\zeta_{\Lambda}$ obtained on the interval $[0,1+\Lambda]$ when $(x_0,y_0)=(1+\Lambda,1)$, for it is the asymptotic completion curve mentionned in  Theorem \ref{mainprofile}.
In this section, we prove that $\zeta   $ satisfies 
$$\left\{ \left( x,\zeta   \left( x\right) \right) \vert\ 0<x\leq x_{0},\zeta   \left( x_{0}\right) =y_{0}\right\} \subset A=\left\{ 0 < y < x\right\},
$$
and stays away  from $\partial A$, if $x$ is large enough, as shown in Figure \ref{zoneuneu}.  This follows from the variations of $\lambda$ along the curve $y=\zeta   (x)$, where we have:
$$\lambda(x)=\dfrac {x}{\zeta   (x)}-1.$$
\begin{figure}[ht]
\centering
\includegraphics[width=16cm]{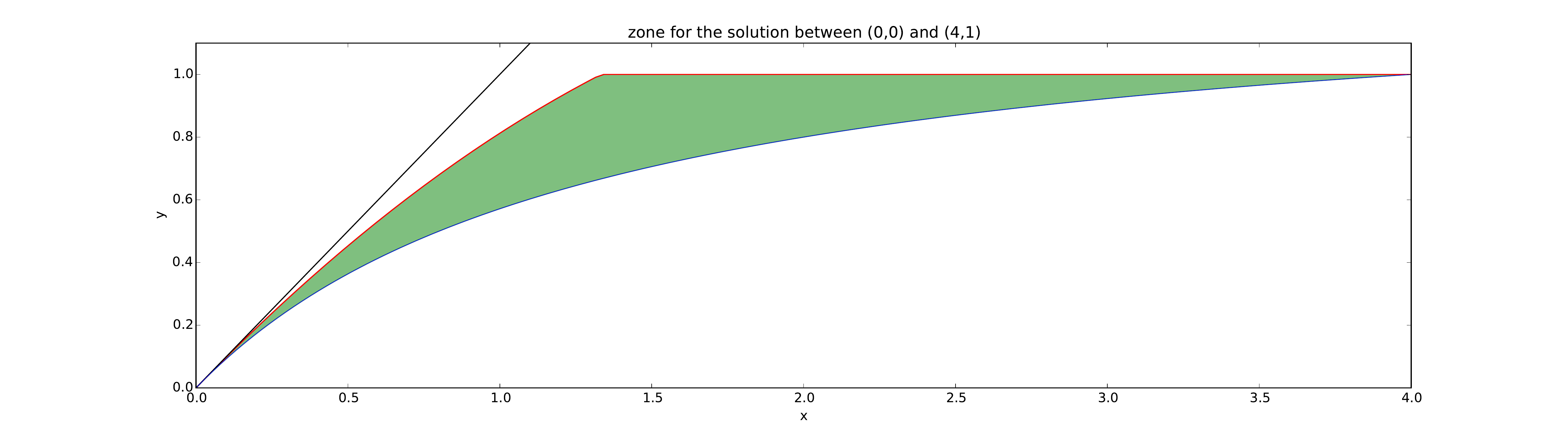}
\\
\includegraphics[width=8cm]{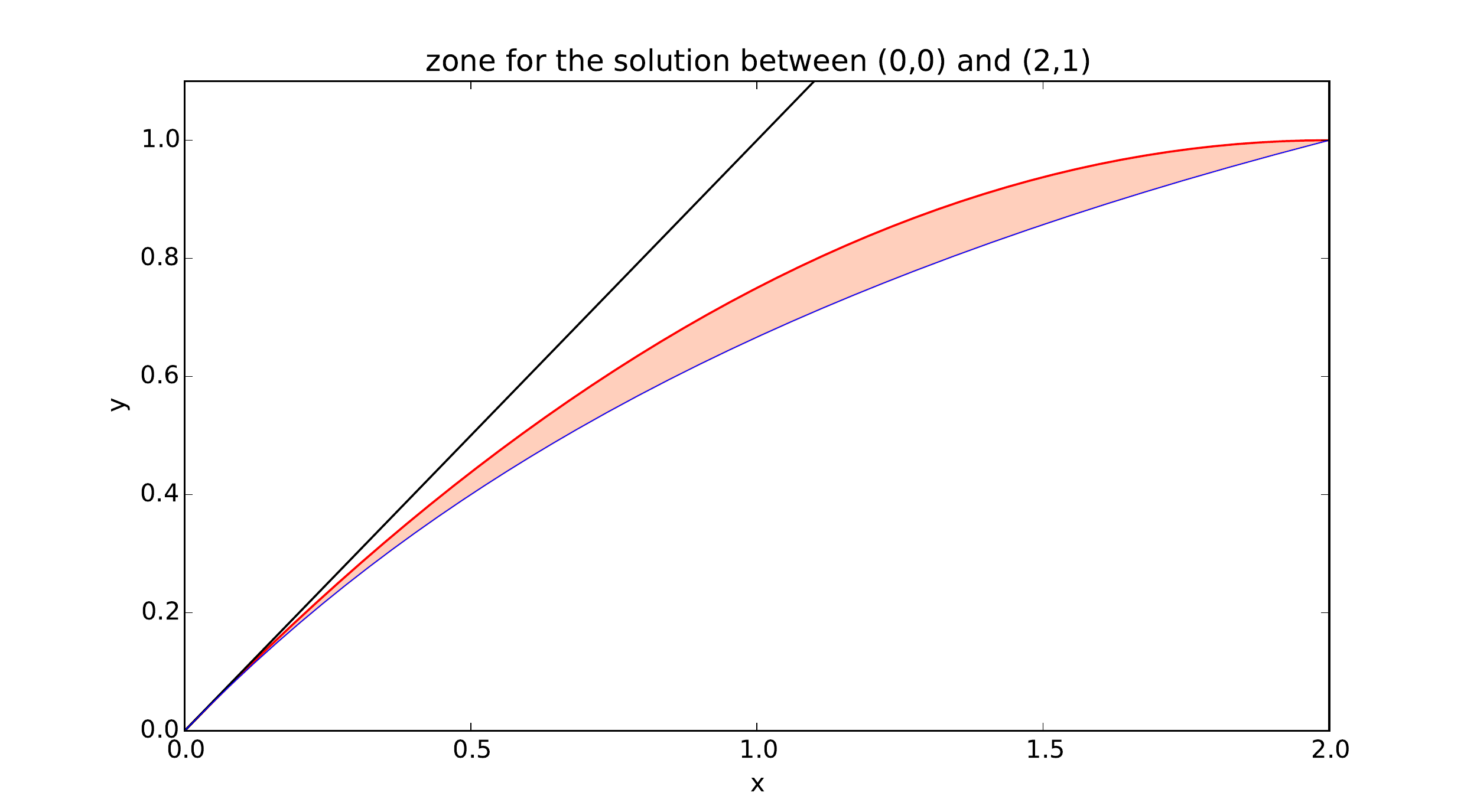}
\caption{Locations of the solutions.}
\label{zoneuneu}
\end{figure}
\begin{prop}
\label{lowerbounded}
The solution $\zeta   $ to \eqref{edfond} going through $(x_0,y_0)$ satisfies, for $ 0< x\le x_0$,
\begin{align}
\label{lambdabound}
-1+\dfrac {1}{1-\dfrac {xy_{0}}{ x_{0}}\left(\dfrac {1}{y_{0}}-\dfrac {1}{x_{0}}\right) }  &\leq\lambda \left( x\right) \leq x\left(\dfrac {1}{y_{0}}-\dfrac {1}{x_{0}}\right),
\\
\label{lowerboundedeq} 
\dfrac {x}{1+x\left( \dfrac {1}{y_{0}}-\dfrac {1}{x_{0}}\right) }&\le  \zeta   (x)\leq x\left( 1-\dfrac {x}{x_{0}}\left( 1-\dfrac {y_{0}}{x_{0}}\right) \right).
\end{align}
\end{prop}
\begin{proof}
From \eqref{edfond},  we obtain the differential equation for $\lambda(x)$  :
\begin{align}
\nonumber
\zeta   (x)&=\dfrac {x}{1+\lambda(x) },
\\
\nonumber
\dfrac{\partial\lambda}{\partial x}&=\dfrac {\zeta   (x)-x\zeta   '(x)}{\zeta   (x)^{2}}=\dfrac {1-\left( 1+\lambda(x) \right) \zeta   '(x)}{x}\left( 1+\lambda \right)
\\
\nonumber
\dfrac{\partial\lambda}{\partial x}&=
\dfrac {1-\left( 1+\lambda \right) e^{-\impl}}{x}\left( 1+\lambda \right) 
\\
\label{ed2}
&=\dfrac {2v(x)}{x}.
\end {align}
{\bf Lower bounds. } Relations \eqref{ed2} and   \eqref{ineqed1} yields that
$$\dfrac {\lambda '}{\lambda }\geq \dfrac {1}{x},
$$
thus, for $ 0< x\le x_0$,
\begin{equation}
\label{decrease}
\lambda \left( x\right) \leq \lambda \left( x_{0}\right) \dfrac {x}{x_{0}} =\dfrac {x}{y_{0}}-\dfrac {x}{x_{0}},
\end{equation}
leading to the lower bounds of \eqref{lowerboundedeq}.

{\bf Upper bounds. } With \eqref{zetabound1} and \eqref{ed2} together, we obtain: 
\begin{align*}
		\dfrac {\lambda '}{\lambda \left( 1+\lambda \right) }&\leq \dfrac {1}{x},
		\\
		\dfrac {\lambda \left( x_{0}\right) \left( 1+\lambda \left( x\right) \right) }{\left( 1+\lambda \left( x_{0}\right) \right) \lambda \left( x\right) }&\leq \dfrac {x_{0}}{x},
\end{align*}
leading to the upper bounds in Proposition \ref{lowerbounded}.
\end{proof}

These estimates are also useful in the proof of Kor\v{s}unov's formula, in which we need that a strip close to the limit path intersects the forbidden zone $\{y\le x/k\} $ only close to its endpoints $(0,0)$ and $(k,1)$. We have :
\begin{prop}
	\label{kielbasa}
For $0<\varepsilon\le \dfrac{1}{2(k+1)}$, and $2k\varepsilon\le x\le k-2k^2\varepsilon$,
$$f_{k}\left( x\right) -\varepsilon \geq \dfrac {x}{k}.$$
\end{prop}
\begin{proof}
Due to \eqref{lowerboundedeq}, we only need to prove that
$$\dfrac {x}{1+x\left( 1-\dfrac {1}{k}\right) }-\dfrac {x}{k}\ge \varepsilon,$$
when $x$ is in the interval, and it follows easily from the fact that
$$\left( k-1\right) x\left( k-x\right) \geq \varepsilon k\left( k+x\left( k-1\right) \right) $$
holds true at the endpoints.
\end{proof}

\subsection{Small variations of $\rho$}
\label{sec:deltarho}
In this section, we  bound  the variations of $\rho$ in order to obtain the accuracy of the Euler scheme used in  Theorem \ref{mainprofile}, cf. \eqref{boundsegonde}, and also to obtain the precision of the approximation of the completion curve by a random walk in the proof of Kor\v{s}unov's formula.
Since $\rho=e^{-\impl}$, according to \eqref{derive}
\begin{align*}
\dfrac{\partial\rho}{\partial\lambda}&=-\dfrac{\partial}{\partial\lambda}\left( 1- e^{-\impl }\right)
=-\dfrac{\partial}{\partial\lambda}\left( \dfrac{\impl}{1+\lambda}\right)
\\
&=\dfrac {-(1+\lambda )\impl'+\impl}{(1+\lambda )^2}
\\
&=\dfrac {\impl(\impl-\lambda-1)}{(\impl-\lambda )(1+\lambda )^2}
\end{align*}
Thus, according to \eqref{ed2},  a sample path $\zeta   $ solution of \eqref{edfond} satisfies
\begin{align*}
0\ge\zeta   ^{\prime\prime}=\dfrac{\partial\rho}{\partial x}
&=\dfrac {\impl(\impl-\lambda-1)}{(\impl-\lambda )(1+\lambda )^2}\times\dfrac {(\impl-\lambda )(1+\lambda )}x
\\
&=\dfrac {\impl(\impl-\lambda-1)}{x(1+\lambda )}\ \ge\ -\dfrac {1}{x}.\end{align*}
Similarly, anywhere in the domain,
\begin{align*}
|F'_{y}|
&=\left\vert -\dfrac {\impl(\impl-\lambda-1)}{(\impl-\lambda )(1+\lambda )^2}\times\dfrac {1+\lambda }y\right\vert
\\
&=\dfrac {\impl(\lambda+1-\impl)}{(\impl-\lambda )x}\ \le\ \dfrac {2}{x},
\\
|F'_{x}|
&=\left\vert \dfrac {\impl(\impl-\lambda-1)}{(\impl-\lambda )(1+\lambda )^2}\times\dfrac {1}y\right\vert
\ \le\ \dfrac {2}{x},\end{align*}
since $\dfrac {\impl(\lambda+1-\impl)}{(\impl-\lambda )}\le 2$ holds true, for it reduces to $2(e^{\impl }-1-\impl)\ge\impl^2$.
Thus we have:
\begin{prop}
\label{prop-deltarho}
If $\{(x,\zeta   (x)),(x,y)\}\subset\coin_{\eta,\delta}$,
$$|\zeta   ''(x)|\le\dfrac {1}{\eta}\quad\text{and}\quad|F'_{y}(x,y)|\le\ \dfrac {2}{\eta},\ |F'_{x}(x,y)|\le\ \dfrac {2}{\eta}.$$
\end{prop}
\subsection{Proof of Theorem \ref{transitionp}}
\label{sec:transition}
\subsubsection{Small variations of $\impl  $}
\label{swimpl}
For $m>\ell\ge2$, $\lambda>0$ and $\tilde\lambda=\tfrac{ m-1}{\ell-1}-1 = \lambda+\tfrac{\lambda}{\ell-1}$, $\impl   =\impl \left( \lambda \right) $, $\tilde\impl   =\impl \left( \tilde\lambda \right) $, we set, for any real function $f$,
$$\Delta(f)=f\left( \tilde{\lambda }\right) -f\left( \lambda \right).$$
Then we have:
\begin{prop}For $m>\ell\ge2$, $\lambda>0$ , we have
\label{Taylor}
\begin{align*}
\left| \Delta ({\impl   }) -\dfrac {\lambda\impl   }{(1+\lambda)(\impl    -\lambda )(\ell-1)}\right|
&\le  \dfrac {\lambda ^{2}}{2\left(\impl    -\lambda \right) ^{3}\left( \ell-1\right) ^{2}},
\\
\left|  \Delta ({\ln\impl   })  -\dfrac {\lambda}{(1+\lambda)(\impl    -\lambda )(\ell-1)}\right|
&\le  \dfrac {\lambda ^{2}}{8\left(\impl    -\lambda \right) ^{3}\left( \ell-1\right) ^{2}},
\\
\left| \Delta \left( \ln \left( 1+\lambda \right) \right) -\dfrac {\lambda }{\left( 1+\lambda \right) \left( \ell-1\right) }\right| 
&\le \dfrac {\lambda ^{2}}{2\left( 1+\lambda \right) ^{2}\left( \ell-1\right) ^{2}}.
\end{align*}
\end{prop}
\begin{proof} We need a bound for
\begin{align*}
\vert\impl ''\left( \lambda \right)\vert&= \dfrac {\impl }{1+\lambda }\   \dfrac {2\lambda -\impl }{1+\lambda }\  \dfrac {1+\lambda -\impl }{\left( \impl -\lambda \right) ^{3}}\\
&\le \dfrac {1}{\left( \impl -\lambda \right) ^{3}}.
\end{align*}
Thus,
\begin{align*}
\left| \tilde {\impl   }-\impl   -\dfrac {\lambda }{\ell-1}\impl '\left( \lambda \right) \right|
&\le  \dfrac {1}{\left( \impl(\lambda ) -\lambda \right) ^{3}}\dfrac {\lambda ^{2}}{2\left( \ell-1\right) ^{2}}
\\
&\le  \dfrac {\lambda ^{2}}{2\left(\impl    -\lambda \right) ^{3}\left( \ell-1\right) ^{2}}.
\end{align*}
In order to bound $\left| \ln\tilde {\impl   }-\ln\impl   \right|$, after some computations starting from:
$$\left( \ln \impl \right) ''=\dfrac {\impl '' \impl -\impl '^{2}}{\impl ^{2}},$$
we obtain:
\begin{align*}
\left| \left( \ln \impl \right) ''\right|
&=\left|\dfrac{(\impl-2\lambda )(\impl-\lambda )+\lambda }{(1+\lambda )^2(\impl -\lambda )^3}\right|
\\
&\le\dfrac{1}{4(\impl -\lambda )^3},
\end{align*}
since
\begin{align*}
\left| (\impl-2\lambda )(\impl-\lambda )+\lambda\right|
&\le\max\left((2\lambda-\impl )(\impl-\lambda ),\lambda\right)\le\max\left(\dfrac{\lambda^2}4,\dfrac{(1+\lambda)^2}4\right),
\end{align*}
yielding a bound, for the second  derivative, that entails the desired result.
\end{proof}

\subsubsection{Final argument}
Now we can use Theorem \ref{goodplus} to bound the error $\errtr$ in the approximation of the transition probability $r\left( m,\ell \right)$ by $\rho \left( \lambda\right)=e^{-\impl}$:
$$\left| r\left( m,\ell \right)-\rho \left( \lambda\right)  \right| =\errtr \left( m,\ell \right).$$
Set :
$$\tilde{r}\left( m,\ell\right) =\dfrac {\psi \left( m-1,\ell-1\right) }{\psi \left( m,\ell\right) },$$
so that :
\begin{align*}
\errtr \left( m,\ell \right) &\leq \left|r\left( m,\ell \right) -\tilde{r}\left( m,\ell \right) \right|+\left|\tilde{r}\left( m,\ell \right)-\rho \left( \dfrac {m-\ell }{\ell }\right) \right| 
\\
&=\errtr_{1}\left( m,\ell \right)+\errtr_{2}\left( m,\ell \right).
\end{align*}
First :
\begin{align*}
\errtr_1 \left( m,\ell \right) &\le \dfrac {\psi \left( m-1,\ell -1\right) }{
{ m\brace\ell} }\left| \chi\left( m,\ell \right) -\chi\left( m-1,\ell -1\right) \right|.
\end{align*}
Since
$$0\le \dfrac {{ m-1\brace\ell -1}}{{ m\brace\ell} }\le 1,$$
we have
\begin{align*}
\dfrac {\psi \left( m-1,\ell -1\right) }{{ m\brace\ell} }&\le \dfrac {\psi \left( m-1,\ell-1\right) }{{ m-1\brace\ell-1} }
\\
&=\dfrac {1}{1+ \chi\left( m-1,\ell-1 \right)}.
\end{align*}
According to Theorem \ref{goodplus}, 
$$\chi\left( m-1,\ell-1 \right) =\ \mathcal{O}\left(\frac{1}{\ell }\right),$$
thus, for $\ell$ large enough, $\chi\left( m-1,\ell-1 \right) \ge -\tfrac12$ and
\begin{align*}
\errtr_1 \left( m,\ell \right) &\le 2|\chi\left( m,\ell \right)| +2|\chi\left( m-1,\ell -1\right)|=\ \mathcal{O}\left(\frac{1}{\ell }\right),
\end{align*}
Now, for some $ u\in[0, 1]$, 
\begin{align*}
\errtr_{2}\left( m,\ell \right)&=e^{-\impl   }\left( e^{\theta \left( m,\ell\right) }-1\right) =e^{-\impl   +u\theta \left( m,\ell\right) } \theta \left( m,\ell\right) 
\end{align*}
in which $\theta \left( m,\ell\right) $, that turns out to be $\mathcal{O}\left(\frac{1}{\ell }\right)$, is defined as follows :
\begin{align*}\theta \left( m,\ell\right) 
&=\impl  +\ln \psi \left( m-1,\ell -1\right) -\ln \psi \left( m,\ell \right) .
\end{align*}
We write
\begin{align*}
\theta \left( m,\ell\right) &=A+B,
\end{align*}
with
\begin{align*}
A&=\ln\left( \dfrac {e^{\tilde\impl  }-1}{e^{\impl  }-1}\right) ^{\ell }\left( \dfrac {\impl   }{\tilde\impl   }\right) ^{m}
\\
&=\ell\left( \ln\left( \dfrac {e^{\tilde\impl  }-1}{e^{\impl  }-1}\right)\left( \dfrac {\impl   }{\tilde\impl   }\right) ^{1+\lambda}\right)
\\-
&=\ell\left( \ln\left( \dfrac {e^{\tilde\impl  }(1+\lambda)}{e^{\impl  }(1+\tilde\lambda)}\right)\left( \dfrac {\impl   }{\tilde\impl   }\right) ^{\lambda}\right),
\\
\dfrac {A}{\ell }&=\Delta \left(\impl  \right)-\lambda\Delta\left(\ln\impl   \right)-\Delta \left(\ln (1+\lambda)\right),
\\
2B&= -\Delta \left(2\impl  +\ln (\impl   -\lambda) \right) + \Delta \left( \ln (1+\lambda )\right) -\ln \left( 1-\dfrac {1}{ \ell}\right) .
\end{align*}
The factor $\ell$ in $A$ is the reason why we need the second order approximations of Section \ref{swimpl}. Now we see, from Proposition \ref{Taylor}, that :
\begin{align*}
\lim _{\ell}B&= 
\lim _{\ell}A=0,
\end{align*}
so that
$$\lim _{\ell}r\left(m,\ell\right)-\rho\left(\tfrac{m-\ell}{\ell}\right)=0.$$
But, more precisely, Proposition \ref{Taylor} entails:
$$\left| \dfrac {A}{\ell}\right| \le\dfrac {\lambda ^{2}}{2\left( 1+\lambda \right) ^{2}\left( \ell-1\right) ^{2}}\left( 1+\left(\dfrac { 1+\lambda }{\impl   -\lambda }\right) ^{3}\right),
$$
that is, $A=\mathcal{O}\left(\tfrac{1}{\ell }\right)$. Now, for $B$, since $\lambda \rightarrow \impl \left( \lambda \right)$, $ \lambda \rightarrow \impl \left( \lambda \right) -\lambda $, $\lambda \rightarrow \ln \left( 1+\lambda \right) $, are increasing and concave, then $\lambda \rightarrow \ln \left( \impl \left( \lambda \right) -\lambda \right) $ is increasing and concave too, being composed with an increasing and concave function, so, due to Taylor-Lagrange formula, all these functions satisfy:
$$\left| \Delta f\right| \leq \dfrac {\lambda f'\left( \lambda \right) }{\ell-1},$$
and that yields:
\begin{align*}
\left| \Delta \impl   \right| 
&\le \dfrac {\impl   \lambda }{\left( 1+\lambda \right) \left( \impl   -\lambda \right) \left( \ell-1\right) },
\\
\left| \Delta \ln \left( 1+\lambda \right) \right|
&\le 
\dfrac {\lambda }{\left( 1+\lambda \right) \left( \ell-1\right) },
\\
\left| \Delta \ln \left( \impl   -\lambda \right) \right| 
&\le 
\dfrac {\lambda ^{2}e^{-\impl   }}{\left( \impl   -\lambda \right) ^{2}\left( \ell-1\right) }.
\end{align*}
Also, for $ \ell\ge\tfrac12$, 
$$\left| \ln \left( 1-\dfrac1\ell \right) \right| \leq \dfrac{2 \ln 2}\ell,$$
so that, using \eqref{ineqed1},
\begin{align*}
\left| B \right|
 &\le
 \max\Big( \dfrac {\impl   \lambda }{\left( 1+\lambda \right) \left( \impl   -\lambda \right) \left( \ell-1\right) }
\\
&\hspace{2.5cm}+\dfrac {\lambda ^{2}e^{-\impl   }}{2\left( \impl   -\lambda \right) ^{2}\left( \ell-1\right) },
\dfrac {\lambda }{2\left( 1+\lambda \right) \left( \ell-1\right) }+\dfrac {\ln2 }{\ell}\Big) 
\\
&\le \max\left( \dfrac {\impl   }{ \ell-1 }
+\dfrac {\left( 1+\lambda \right) ^{2}e^{-\impl   }}{2\left( \ell-1\right) },\dfrac {1+2\ln2 }{2\left( \ell-1\right) }\right),
\\
&\le \max\left( \dfrac {\impl   }{ \ell-1 }
+\dfrac {1+\lambda }{2\left( \ell-1\right) },\dfrac {3}{2\left( \ell-1\right) }\right)\le \dfrac {3\left( 1+\lambda \right)}{2\left( \ell-1\right)} ,
\end{align*}
and, using \eqref{ineqed1} again, for $(m,\ell)\in \coin_{3, \delta}$,
\begin{align*}
|\theta \left( m,\ell\right)|\le\left| A+B  \right|
 &\le
\dfrac {3\left( 1+\lambda \right)}{2\left( \ell-1\right)}+\dfrac {\lambda ^{2}\ell}{2\left( 1+\lambda \right) ^{2}\left( \ell-1\right) ^{2}}\left( 1+\dfrac { (1+\lambda)^{6} }{\lambda ^{3}} \right)
\\
&\le
\dfrac {3\left( 1+\lambda \right)}{2\left( \ell-1\right)}+\dfrac { 21(1+\lambda)^{4}\ell}{40\lambda\left( \ell-1\right) ^{2}}
\\
&\le
\dfrac {3\left( 1+\lambda \right)}{2\left( \ell-1\right)}+\dfrac {63 (1+\lambda)^{4}}{80\lambda\left( \ell-1\right) }
\\
&\le
\dfrac {6\left( 1+\lambda \right) ^{4}}{5\lambda\left( \ell-1\right)}\le
\dfrac {6\left( 1+\delta \right) ^{4}}{5\delta^{3}\left( \ell-1\right)}\le
\dfrac {20}{\delta^{3}\ell},
\end{align*}
so that, for $\delta\le\lambda\le\delta^{-1}$ and $\ell\ge 20\delta^{-3}$, we have $|\theta|\le 1$ and 
\begin{align*}
|\errtr_{2}\left( m,\ell \right)|&=e^{-\impl   +u\theta \left( m,\ell\right) } |\theta \left( m,\ell\right)|
\\
&\le e^{|\theta \left( m,\ell\right)|} |\theta \left( m,\ell\right)|
\\
&\le \dfrac {20e}{\delta^{3}\ell}. 
\end{align*}
For instance, this holds true for $(m,\ell)\in \coin_{40\delta^{-4}, \delta}$.

\subsection{Explicit bounds for the second order asymptotics of ${m\brace\ell}$}
\label{sec:taylor}
In this section, we provide detailed computations in order to bound $|\chi(m,\ell)|$, thus completing the proof of  Theorem \ref{goodplus}.
\subsubsection{Taylor coefficients}
As usual, the derivatives of a characteristic function such as $\Phi$ are  bounded as follows:
\begin{equation*}\left| \Phi ^{\left( k\right) }\left( \theta \right) \right| =\left| \mathbb{E} \left[ \left( iZ\right) ^{k}e^{i\theta Z}\right] \right| \leq \mathbb{E} \left[ Z^{k}\right],
\end{equation*}
thus, due to \eqref{poissoncar}, we need the first moments of the Poisson distribution, given by the Touchard polynomials:
\begin{align}
\nonumber
\left( \mathbb{E} \left[ Z^{k}\right] \right) _{1\leq k\leq 5}&=\big( \impl   ,\impl   ^{2}+\impl   ,\impl   ^{3}+3\impl   ^{2}+\impl  ,  \impl  ^4+6\impl   ^{3}+7\impl   ^{2}+\impl   ,
\\
\label{momentpoisson}
&\hspace{4cm}\impl  ^5+10\impl   ^{4} +25\impl   ^{3} +15\impl   ^{2}+\impl   \big) ,
\end{align}
in order to compute the coefficients in the Taylor-Laplace formula for $g$,
for the derivatives of $g$ are obtained through the Leibniz rule, as follows:
\begin{align*}
g(\theta)&=\dfrac {e^{-i\left( 1+\lambda \right) \theta }}{1-e^{-\impl  }}\left(\Phi(\theta)-e^{-\impl  }\right),
\\
g'(\theta)&=\dfrac {e^{-i\left( 1+\lambda \right) \theta }}{1-e^{-\impl  }}\left(-i\left( 1+\lambda \right)\left(\Phi(\theta)-e^{-\impl  }\right)+\Phi'(\theta)\right),
\\
g''(\theta)&=\dfrac {e^{-i\left( 1+\lambda \right) \theta }}{1-e^{-\impl  }}\left(-\left( 1+\lambda \right)^{2}\left(\Phi(\theta)-e^{-\impl  }\right)-2i\left( 1+\lambda \right)\Phi'(\theta)+\Phi''(\theta)\right),
\\
g^{(3)}(\theta)&=\dots
\\
g^{(4)}(\theta)&=\dots 
\\
g^{(5)}(\theta)&=\dfrac {e^{-i\left( 1+\lambda \right) \theta }}{1-e^{-\impl  }}\Big(-i\left( 1+\lambda \right)^{5}\left(\Phi(\theta)-e^{-\impl  }\right)+5\left( 1+\lambda \right)^{4}\Phi'(\theta)+10i\left( 1+\lambda \right)^{3}\Phi''(\theta)\\
&\hspace{2cm}-10\left( 1+\lambda \right)^{2}\Phi^{(3)}(\theta)-5i\left( 1+\lambda \right)\Phi^{(4)}(\theta)+\Phi^{(5)}(\theta)\Big).
\end{align*}
This gives the coefficients in the Taylor-Laplace inequality:
\begin{align*}
&\left| g\left( \theta \right) -g\left( 0\right) -g'\left( 0\right) \theta -g''\left( 0\right) \dfrac {\theta ^{2}}{2}-g^{(3)}\left( 0\right) \dfrac {\theta ^{3}}{6}-g^{(4)}\left( 0\right) \dfrac {\theta ^{4}}{24}\right| \\
&\hspace{4cm}=\left| \int ^{\theta }_{0}\dfrac {g^{\left( 5\right) }\left( u\right) }{5!}\left( \theta -u\right) ^{5}du\right| 
\le \dfrac {\theta ^{5}}{120}\sup _{\left[ 0,\theta \right] }\left| g^{\left( 5\right) }\left( u\right) \right|,
\end{align*}
that is:
\begin{align*}
g(0)&=1,
\\
g'(0)&=\dfrac {1}{1-e^{-\impl  }}\left(-i\left( 1+\lambda \right)\left(1-e^{-\impl  }\right)+i\impl  \right),
\\
&=\dfrac {1}{1-e^{-\impl  }}\left(-i\impl  +i\impl  \right)=0,
\\
g''(0)&=\dfrac {1+\lambda}{\impl  }\left(-\left( 1+\lambda \right)^{2}\left(1-e^{-\impl  }\right)+2\left( 1+\lambda \right)\impl  -\impl  -\impl  ^{2}\right)=-2v,
\\
&=\left( 1+\lambda \right)\left(\lambda-\impl  \right),
\\
g^{(3)}(0)&=6\tau=i\left( 1+\lambda \right) \left( \impl   ^{2}+3\impl   \lambda +\lambda \left( 1+2\lambda \right) \right),
\\
g^{(4)}(0)&=24\gamma=\left( 1+\lambda \right) \left( \impl   ^{3}+\impl   ^{2}\left( 3-4\lambda \right) +6\lambda ^{2}\impl   -\lambda \left( 3\lambda ^{2}+3\lambda +1\right) \right),
\end{align*}
computations needed in order to bound the coefficients in \eqref{affine1}. 
\subsubsection{Upper bound for $C_{\texttt{000}}$}
\label{beeeurk}
Finally, for $\theta\in\R$, the fifth derivative is bounded as follows:
\begin{align*}
\left| g^{(5)}(\theta) \right|
&\le\dfrac {1+\lambda}{\impl  }\Big(\left( 1+\lambda \right)^{5}\left|\Phi(\theta)-e^{-\impl  }\right|+5\left( 1+\lambda \right)^{4}\impl  +10\left( 1+\lambda \right)^{3}\impl  \left( \impl   +1\right)
\\&\hspace{1cm}+ 10\left( 1+\lambda \right)^{2}\impl  \left( \impl   ^{2}+3\impl   +1\right)+ 5\left( 1+\lambda \right)\impl  \left( \impl   ^{3}+6\impl   ^{2}+7\impl   +1\right)
\\&\hspace{3cm}+ \impl  \left( \impl   ^{4}
+10 \impl   ^{3}+25\impl   ^{2}+15\impl   +1\right)\Big)
\\
&\le\dfrac {2\left( 1+\lambda \right)^{6}}{\lambda}+\left( 1+\lambda \right)\Big(5\left( 1+\lambda \right)^{4}+10\left( 1+\lambda \right)^{3}\left(\lambda +2\right)
\\&\hspace{1cm}+ 10\left( 1+\lambda \right)^{2}\left( \impl   ^{2}+3\impl   +1\right)+ 5\left( 1+\lambda \right)\left( \impl   ^{3}+6\impl   ^{2}+7\impl   +1\right)
\\&\hspace{3cm}+ \left( \impl   ^{4}
+10 \impl   ^{3}+25\impl   ^{2}+15\impl   +1\right)\Big)
\\
&\le\dfrac {12\left( 1+\lambda \right)^{6}}{\lambda}+\left( 1+\lambda \right)\Big(5\left( 1+\lambda \right)^{4}+ 10\left( 1+\lambda \right)^{2}\left( \impl   ^{2}+3\impl   +1\right)
\\&\hspace{1cm}+ 5\left( 1+\lambda \right)\left( \impl   ^{3}+6\impl   ^{2}+7\impl   +1\right)
+ \left( \impl   ^{4}
+10 \impl   ^{3}+25\impl   ^{2}+15\impl   +1\right)\Big)
\\
&\le\dfrac {46\left( 1+\lambda \right)^{6}}{\lambda}=C_{4}\left(\lambda \right),
\end{align*}
in which we use again and again $\lambda \leq \impl   \leq 1+\lambda $ and $ \lambda \left( 2+\lambda \right) \leq \left( 1+\lambda \right) ^{2}$, cf. Figure \ref{calculs201810}.

\begin{figure}[ht]
\centering
\includegraphics[width=11cm]{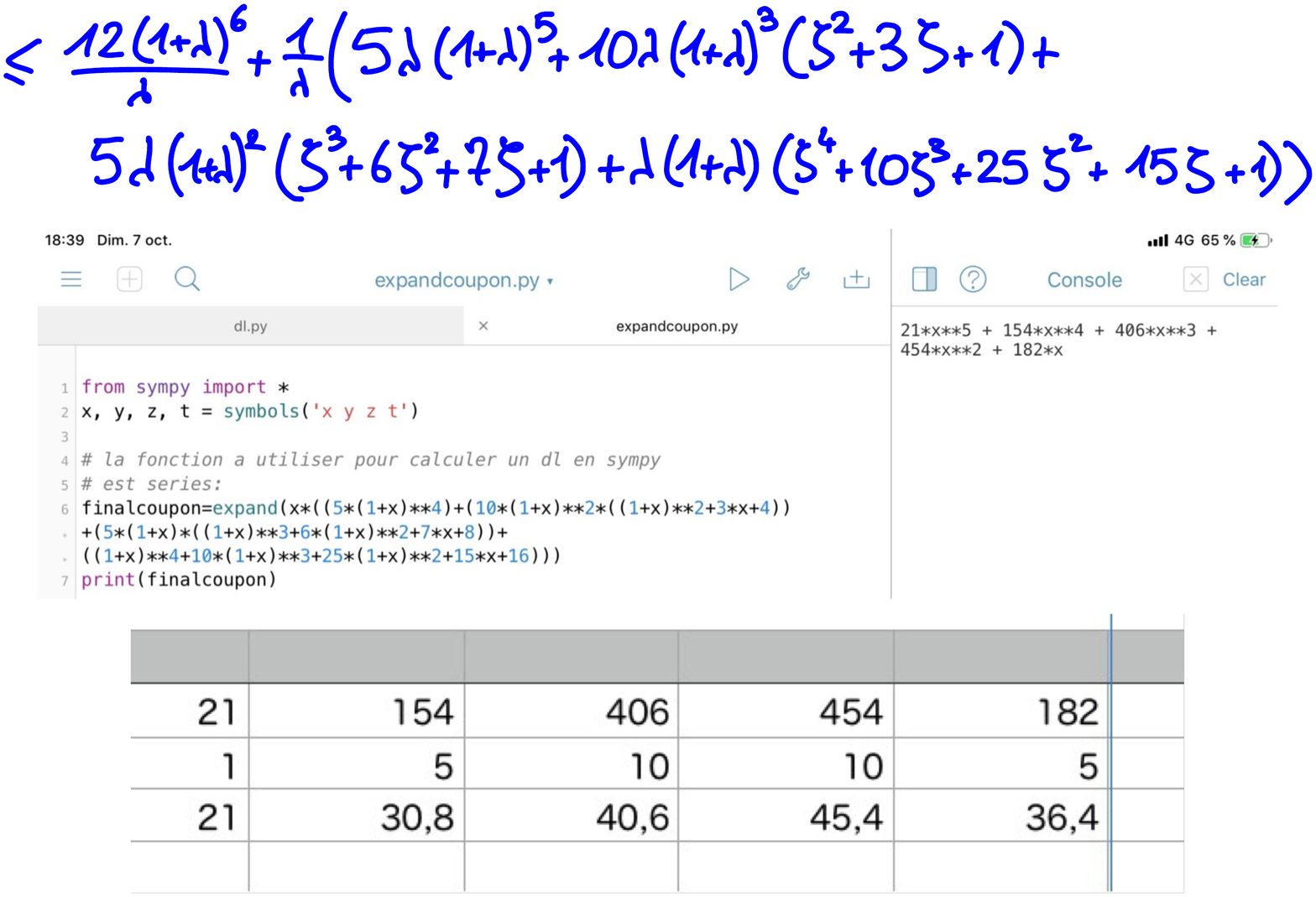}
\caption{Computations: $46\ge 45.4$.}
\label{calculs201810}
\end{figure}

Thus we just proved that:
\begin{align}\label{affine6}
\left| g\left( \theta \right) -\left( 1-v\theta ^{2}+\tau \theta ^{3}+\gamma \theta ^{4}\right)\right| \leq T\left(\lambda \right) |\theta|^{5},
\end{align}
in which
$T\left(\lambda \right)=C_{4}\left(\lambda \right)/120$.
This leads to 
$C_{\texttt{000}}\le T\left(\lambda \right)/3,$
contributing to the bound for $|\chi(m,\ell)|$ through 
\begin{align*}
\sqrt{\frac{ v\ell}{\pi}}\ K_{\ell}^{(000)}&\le \sqrt{\frac{ v\ell}{\pi}}\  C_{\texttt{000}}\ \ell\ \theta_0^{6}\le \dfrac {\left( 1+\lambda \right)^{7}}{10\lambda}\ \ln( \ell)^{6}\ \ell^{-3/2},
\end{align*}
so that $\lambda \in(\delta,\delta^{-1})$ entails
\begin{align*}
\sqrt{\frac{ v\ell}{\pi}}\ K_{\ell}^{(000)}&\le 7\,\delta^{-8}\ \ln( \ell)^{6}\ \ell^{-3/2}.
\end{align*}

In the next sections, we shall also use the following inequalities:
\begin{align}
\label{ineqcoeffs}
\left| \tilde{\gamma} \right| &\leq \dfrac {13}{24}\left( 1+\lambda \right) ^{4},
\quad
\left| \gamma \right| &\leq \dfrac {7}{24}\left( 1+\lambda \right) ^{4},
\quad
\left| \tau \right|\leq \left( 1+\lambda \right) ^{3},
\quad
\lambda \leq 2v\leq 1+\lambda.
\end{align}
\subsubsection{Upper bound for $C_{\texttt{001}}$}
The choice $\tilde{\gamma}=\gamma-\tfrac{v^{2}}2$ insures that \begin{align*}
\kappa_{7}(\theta)&=1-v\theta^2+\tau\theta^3+\gamma\theta^4-e^{-v\theta^2+\tau\theta^3+\tilde{\gamma}\theta^4}= \mathcal{O}\left(\theta^5\right),
\end{align*}
so that
\begin{align*}
\sup\left\{|\kappa_{7}(\theta)||\theta|^{-5},\quad |\theta|\le\theta_0\right\}=C_{5}<+\infty .
\end{align*}
Then $C_{\texttt{001}}=C_{5}/3$, and, as a function of $\lambda$, $C_{\texttt{001}}$ is bounded  for $\lambda\in[\delta,\delta^{-1}]$, for any choice of $\delta\in (0,1)$, just like  $T\left(\lambda \right)$, $C_{\texttt{000}}$, and the other coefficients in relation \eqref{affine1}. More precisely, for $C_{5}$, for instance, we have
$$\left| \dfrac {\tau }{v}\right| \leq \dfrac {\left( 1+\lambda \right) ^{3}}{2\lambda },\qquad \left| \dfrac {\tilde{\gamma} }{v}\right| \leq \dfrac {13\left( 1+\lambda \right) ^{4}}{12\lambda }.$$ 
For $\delta\le \lambda\le \delta^{-1}$, and $\ell$ large enough, we have $\theta_0=\dfrac {\ln\ell}{\sqrt\ell}\le\dfrac{\delta^{3}}{2\left( 1+\delta \right) ^{4}}$, thus $|\theta|\le\theta_0$ entails that
\begin{align*}
\Re(u)&=\Re\left(-v\theta^2+\tau\theta^3+\gamma\theta^4\right)
\\
&=-v\theta^2\ \Re\left(1-\tfrac{\tau}v\theta-\tfrac{\tilde{\gamma}}v\theta^2\right)\le 0,
\end{align*}
and , as a consequence, 
$$ \left| e^{-u}-1+u-\dfrac {u^{2}}{2}\right| \leq \dfrac {\left| u\right| ^{3}}{6}.
$$
Then 
$$|\kappa_{7}(\theta)|\le \dfrac {\left| u\right| ^{3}}{6}+\dfrac {1}{2}\left|u^{2}-v^2\theta^4\right|,$$
but, since $\theta_0 \le 1 $,
$$\left| u\right|^{3}  \le \left( v+\left| \tau \right| +\left| \gamma \right|  \right) ^{3}\theta ^{6} \le3 \left( 1+\lambda \right) ^{12}\theta ^{6},$$
and
$$\left|u^{2}-v^2\theta^4\right|\le3 \left( 1+\lambda \right) ^{8}\theta ^{5}.$$
Finally 
$$|\kappa_{7}(\theta)| \le
2 \left( 1+\lambda \right) ^{12}\theta ^{5},$$
and $C_{\texttt{001}}=\left( 1+\lambda \right) ^{12}$ does the trick. Finally, $\lambda \in(\delta,\delta^{-1})$ entails that the corresponding contribution to  $|\chi(m,\ell)|$ is bounded as follows :
\begin{align*}
\sqrt{\frac{ v\ell}{\pi}}\ K_{\ell}^{(001)}\le\sqrt{\frac{ v\ell}{\pi}}\ C_{\texttt{001}}\,\dfrac{\ln^{6}\ell}{\ell^{2}}&\le 2^{12}\,\delta^{-13}\ \ln( \ell)^{6}\ \ell^{-3/2}.
\end{align*}

\subsubsection{Upper bound for $C_{\texttt{0020}}$ and for $C_{\texttt{0021}}$}
First, inequality \eqref{0020} holds true, for instance, when 
$$\ell\,\ln^{-2} \ell\ge 16\,\delta^{-2},$$
 while inequality \eqref{002} holds true if $\ell|\tilde{\gamma}|\theta_0^4\le \ln2$, for instance if 
$$\ell\,\ln^{-4} \ell\ge16\,\delta^{-4}/\ln2.$$
Then
$$2\sqrt{\pi}\,v^{-7/2}\left(-\tau^2+\tilde{\gamma}^2\right)\   \ell^{-5/2}\le\  C_{\texttt{0020}}\ell^{-5/2}$$
holds true   if one chooses:
$$C_{\texttt{0020}}\sqrt{\pi/ v}\  =24\left( 1+\lambda \right) ^{8}\lambda^{-3}.$$
Finally, $\lambda \in(\delta,\delta^{-1})$ entails that the corresponding contribution to  $|\chi(m,\ell)|$ is bounded as follows :
\begin{align*}
\sqrt{\frac{ v\ell}{\pi}}\ K_{\ell}^{(0020)}&\le 3\times 2^{11}\,\delta^{-11}\ \ell^{-2}.
\end{align*}
Similarly
\begin{align*}
\sqrt{\frac{ v\ell}{\pi}}\ K_{\ell}^{(0021)}
&\le\,|\tilde{\gamma}|v^{-2}\ell^{-1}\le\, 50\ \delta^{-6}\ \dfrac1\ell.
\end{align*}
\subsubsection{Upper bound for $K_\ell ^{(1)}$}
\label{lastbound}
According to \eqref{affine2},
$$\left|K_\ell ^{(1)}\right|\le 2\pi  \ell^{-h \left( \impl  \right) \ln(\ell)}=o\left( \ell^{-3/2}\right),$$
we have
$$\left|K_\ell ^{(1)}\right|\le 2\pi  \ell^{-2 }=o\left( \ell^{-3/2}\right),$$
as desired, provided that:
$$ \ell\geq \exp\left( \dfrac {2}{h\left( \xi \left( \lambda \right) \right) }\right) ,\quad\forall \lambda \in \left(\delta , \delta^{-1}\right). $$
Due to the variations of $h$, $\xi$, this amounts to:
$$ \ell\geq\max\left( \exp\left( \dfrac {2}{h\left( \xi \left(\delta \right) \right) }\right), \exp\left( \dfrac {2}{h\left( \xi \left(\delta^{-1} \right) \right) }\right)\right) . $$
\subsubsection{Conclusion}
\label{sec:concl}
Finally, for $\lambda \in \left(\delta , \delta^{-1}\right)$ and $\ell$ large enough (be more precise), 
$$|\chi(m,\ell)|\le 50\ \delta^{-6}\ \dfrac1\ell +o(\ell^{-1}),$$
more precisely,
$$|\chi(m,\ell)|\le 50\ \delta^{-6}\ \dfrac1\ell +3\times 2^{11}\,\delta^{-11}\ \ell^{-2}+5\,000\,\delta^{-13}\ \ln( \ell)^{6}\ \ell^{-3/2}.$$

\end{document}